\documentclass[twoside,11pt,reqno]{amsart}
\usepackage{amsmath,amssymb,amscd,mathrsfs,epic}

\makeatletter

\@addtoreset{equation}{section}

\newtheorem{Proposition}{Proposition}[section]
\newtheorem{Lemma}[Proposition]{Lemma}
\newtheorem{Theorem}[Proposition]{Theorem}
\newtheorem{Corollary}[Proposition]{Corollary}

\def\bi{\text{\boldmath$i$}}
\def\bj{\text{\boldmath$j$}}

\newcommand{\res}{\operatorname{res}}
\newcommand{\soc}{\operatorname{soc}}
\newcommand{\head}{\operatorname{head}}

\newcommand{\cont}{\operatorname{cont}}

\newcommand{\Z}{\mathbb{Z}}
\newcommand{\0}{{\bar 0}}
\newcommand{\1}{{\bar 1}}
\def\eps{{\varepsilon}}
\def\phi{{\varphi}}

\newcommand{\ga}{\gamma}
\newcommand{\Ga}{\Gamma}
\newcommand{\la}{\lambda}
\newcommand{\La}{\Lambda}
\newcommand{\al}{\alpha}
\newcommand{\be}{\beta}

\newcommand{\si}{\sigma}

\newcommand{\de}{\delta}

\newcommand{\ka}{\kappa}
\newcommand{\T}{\mathcal{T}}
\newcommand{\U}{\mathcal{U}}

\def\Mtype{\mathtt{M}}
\def\Qtype{\mathtt{Q}}

\newcommand{\FF}{{\mathbb F}}

\newcommand{\SSS}{{\sf S}}
\newcommand{\AAA}{{\sf A}}

\newcommand{\HA}{{\hat{\AAA}}}

\renewcommand{\mod}{\bmod \,}
\newcommand{\JS}{\operatorname{JS}}

\newcommand{\fs}{f^{*}}
\newcommand{\pn}{\pi_n}
\newcommand{\pna}{\pi_{n-1}}
\newcommand{\pnb}{\pi_{n-2}}
\newcommand{\pnc}{\pi_{n-3}}

\newdimen\hoogte    \hoogte=12pt
\newdimen\breedte   \breedte=14pt
\newdimen\dikte     \dikte=0.5pt

\newenvironment{Young}{\begingroup
       \def\vr{\vrule height0.89\hoogte width\dikte depth 0.2\hoogte}
       \def\fbox##1{\vbox{\offinterlineskip
                    \hrule height\dikte
                    \hbox to \breedte{\vr\hfill##1\hfill\vr}
                    \hrule height\dikte}}
       \vbox\bgroup \offinterlineskip \tabskip=-\dikte \lineskip=-\dikte
            \halign\bgroup &\fbox{##\unskip}\unskip  \crcr }
       {\egroup\egroup\endgroup}
\def\diagram#1{\relax\ifmmode\vcenter{\,\begin{Young}#1\end{Young}\,}\else%
              $\vcenter{\,\begin{Young}#1\end{Young}\,}$\fi}

\begin{document}

\title[Representations of symmetric and alternating groups]{{\bf Small-dimensional projective representations of symmetric and alternating  groups}}

\author{\sc Alexander S. Kleshchev}
\address
{Department of Mathematics\\ University of Oregon\\
Eugene\\ OR~97403, USA}
\email{klesh@uoregon.edu}

\author{\sc Pham Huu Tiep}
\address
{Department of Mathematics\\
University of Arizona\\ Tucson\\ AZ~85721, USA} 
\email{tiep@math.arizona.edu}

\thanks{1991 {\em Mathematics Subject Classification:} 20C20, 20E28, 20G40.\\
\indent
Research supported by the NSF (grants DMS-0654147 and DMS-0901241). 
}

\begin{abstract}
We classify the irreducible projective representations of symmetric and 
alternating groups of minimal possible and second minimal possible dimensions, 
and get a lower bound for the third minimal dimension. On the way we obtain some  
new results on branching which might be of independent interest. 
\end{abstract}

\maketitle

\section{Introduction}
We denote by  
$\hat\SSS_n$ and $\hat\AAA_n$ the Schur double covers of the symmetric and alternating groups $\SSS_n$ and $\AAA_n$ (see Section \ref{SSDCTGA} for the specific choice we make). The goal of this paper is to describe 
irreducible projective representations of symmetric and alternating groups of minimal possible and second minimal possible dimensions, or, equivalently the faithful 
irreducible representations of $\hat\SSS_n$ and $\hat\AAA_n$ of two minimal possible dimensions. We also get a lower bound for the third minimal dimension. 

Our ground field is an algebraically closed field $\FF$ of characteristic $p\neq 2$. 
If $p=0$, then the irreducible representations of $\hat \SSS_n$ and $\hat \AAA_n$ over $\FF$ are roughly labeled by the strict partitions of $n$, i.e. the partitions of $n$ with distinct parts. To be more precise to each strict partition of $n$, one associates one or two representations of $\hat \SSS_n$ (of the same dimension if there are two) and similarly for  $\hat \AAA_n$. 

Now, when $p=0$, the representations corresponding to the partition $(n)$ are called {\em basic}, while the representations corresponding to the partition $(n-1,1)$ are called {\em second basic}. To define the basic and the second basic representations of $\hat\SSS_n$ and $\hat \AAA_n$ in characteristic $p>0$, one needs to reduce  the first and second basic representations in characteristic zero modulo $p$ and take appropriate composition factors. This has been worked out in detail  by Wales \cite{Wales}. Again, there are one or two basic representations for $\hat\SSS_n$ and one or two basic representations for $\hat\AAA_n$ (of the same dimension if there are two), and similarly for the second basic. 

The dimensions of the basic and the second basic representations have also been computed by Wales \cite{Wales}. 
To state the result, set 
$$
\kappa_n:=
\left\{
\begin{array}{ll}
1 &\hbox{if $p|n$,}\\
0 &\hbox{otherwise.}
\end{array}
\right.
$$
In particular, $\kappa_n=0$ if $p=0$.  Then the dimensions of the basic representations for $\hat\SSS_n$ and $\hat\AAA_n$  are:
$$
a(\hat\SSS_n):=2^{\lfloor \frac{n-1-\kappa_n}{2} \rfloor},\quad
a(\hat \AAA_n):=2^{\lfloor \frac{n-2-\kappa_n}{2} \rfloor}.
$$
The dimensions of the second basic representations for $\hat\SSS_n$ and $\hat\AAA_n$  are:
\begin{align*}
b(\hat\SSS_n)&:=2^{\lfloor \frac{n-2-\kappa_{n-1}}{2} \rfloor}(n-2-\kappa_n-2\kappa_{n-1}),
\\
b(\hat \AAA_n)&:=2^{\lfloor \frac{n-3-\kappa_{n-1}}{2} \rfloor}(n-2-\kappa_n-2\kappa_{n-1}).
\end{align*}

Now we can state our main result. 

\vspace{2mm}
\noindent
{\bf Main Theorem.}
{\em 
Let $n \geq 12$, $G = \hat\SSS_n$ or $\hat\AAA_n$, and $V$ be a faithful irreducible representation of $G$ over $\FF$. If $\dim V < 2b(G)$, then $V$ is either a basic representation (of dimension $a(G)$) or a second basic representation (of dimension $b(G)$). 
}
\vspace{2mm}

The assumption $n\geq 12$ in the Main Theorem is necessary---for smaller $n$ there are counterexamples. On the other hand, this assumption is not very important, since dimensions of all irreducible representations of $\hat \SSS_n$ and $\hat\AAA_n$ are known for $n\leq 11$ anyway, see \cite{ModAtl}. 

We prove the Main Theorem by induction, for which we need to establish some  
new results on branching (see \S\S3--5). These results might be of independent interest. 
We establish other useful results on the way. For example, we find the labels for 
second basic representations in the modular case (see \S3). 
Such labels were known so far only for basic representations. 

The scheme of our inductive proof of the Main 
Theorem is as follows. 
First of all, it turns out that the treatment is much more streamlined if, 
instead of $G$-modules for $G \in \{ \hat\SSS_n,\hat\AAA_n\}$, one works with 
{\it supermodules} over certain {\it twisted groups algebras} $\T_n$ and $\U_n$. 
This framework is prepared in \S2.   
Consider now a faithful irreducible $G$-module $W$
which is neither a basic nor a second basic module. Then there is 
an irreducible $\T_n$-supermodule $V$ such that $W$ is a composition factor 
of the $G$-module $V$.  
We aim to show that the restriction of $V$ 
to a natural subalgebra $\T_m$ 
with $m \in \{n-1,n-2,n-3\}$, contains enough ``large'' composition factors, i.e. 
composition factors which again are neither a basic nor a second basic 
supermodule of $\T_m$. In this case we can invoke the induction hypothesis 
to show that $\dim V$ is at least a certain bound, which guarantees 
that $\dim W \geq 2b(G)$ (cf. \S6). 
Otherwise, our branching results (\S\S4, 5) 
imply that $V$ is labeled by a so-called {\it Jantzen-Seitz partition}, 
in which case we have to restrict $V$ further down to a natural subalgebra $\T_m$ 
with $m \in \{n-6,n-7,n-8\}$, and again show that this restriction contains 
enough large composition factors.       

The Main Theorem substantially strengthens \cite[Theorem A]{KT}, which in turn strengthened \cite{Wagner}, and fits naturally into the program of describing small dimension representations of quasi-simple groups.
For representations of symmetric and alternating groups results along these lines were obtained in \cite{JamesDim} and \cite[Section~1]{BKIrrRes}. For Chevalley groups, similar results can be found in  
\cite{LandazuriSeitz, SeitzZal,GurTiep,BKLowBound,HM,GMST,GT2} and many others. 

{\em Throughout the paper we assume that 
$n\geq 5$,} unless otherwise stated. For small $n$ symmetric and alternating groups are too small to be interesting.

\section{Preliminaries}

We keep the notation introduced in the Introduction. 

\subsection{Combinatorics}
We review combinatorics of partitions needed for projective representation theory of symmetric groups, referring the reader to \cite[Part II]{Kbook} for more details. Let 
$$
\ell:=\left\{
\begin{array}{ll}
\infty&\hbox{if $p=0$,}\\
(p-1)/2&\hbox{if $p>0$;}
\end{array}\right.
\qquad\text{and}\qquad
I:=\left\{
\begin{array}{ll}
\Z_{\geq 0}&\hbox{if $p=0$,}\\
\{0,1,\dots,\ell\}&\hbox{if $p>0$.}
\end{array}\right.
$$
For any $n \geq 0$, a partition $\la=(\la_1,\la_2,\dots)$ of $n$ 
is {\em $p$-strict} if $\la_r=\la_{r+1}$ for some $r$ implies $p\mid\la_r$.
A $p$-strict partition $\la$
is {\em restricted} if in addition
$$
\left\{
\begin{array}{ll}
\la_r-\la_{r+1}< p &\hbox{if $p|\la_r$},\\
\la_r-\la_{r+1}\leq p &\hbox{if $p \nmid \la_r$}
\end{array}
\right.
$$
for each $r \geq 1$. If $p = 0$, we interpret $p$-strict and restricted $p$-strict partitions as {\em strict partitions}, i.e. partitions all of whose non-zero parts are distinct.
Let 
${\mathcal{RP}}_p(n)$ 
denote the set of all restricted $p$-strict partitions
of $n$. 
The {\em $p'$-height} $h_{p'}(\la)$ of $\la\in\mathcal{P}_p(n)$ is:
$$
h_{p'}(\la):=\big|\{r\mid1\leq r\leq n\ \text{and}\ p\nmid\la_r\}\big| \qquad(\la\in {\mathcal{RP}}_p(n)).
$$

Let $\la$ be a $p$-strict partition.
We identify $\la$ with its {\em Young diagram} consisting of certain nodes (or boxes). A node $(r,s)$ is the node in row $r$ and column $s$. 
We use the repeating pattern 
$
0,1,\dots,\ell-1,\ell,\ell-1,\dots,1,0
$
of elements of $I$ to assign 
($p$-){\em contents} to the nodes. 
For example, if $p=5$ then $\la=(16, 11,10,10,9,5,1)\in \mathcal{RP}_5$,
and the contents of the nodes of $\la$ are:
\vspace{1 mm}
$$
\newdimen\hoogte    \hoogte=12pt
\newdimen\breedte   \breedte=14pt
\newdimen\dikte     \dikte=0.5pt
\diagram{
$0$ & $1$ & $2$ & $1$ & $0$ & $0$& $1$ & $2$ & $1$ & $0$ & $0$ & $1$ & $2$ & $1$ & $0$ & $0$\cr
$0$ & $1$ & $2$ & $1$ & $0$ & $0$ & $1$ & $2$ & $1$ & $0$ & $0$ \cr
$0$ & $1$ & $2$ & $1$ & $0$ & $0$ & $1$ & $2$ & $1$ & $0$  \cr
$0$ & $1$ & $2$ & $1$ & $0$ & $0$ & $1$ & $2$ & $1$ & $0$  \cr
$0$ & $1$ & $2$ & $1$ & $0$ & $0$ & $1$ & $2$ & $1$ \cr
$0$ & $1$ & $2$ & $1$ & $0$\cr
$0$ \cr
}
\vspace{1 mm}
$$
\vspace{1 mm}
The content of the node $A$ is denoted by $\cont_p A$. Since the content of the node $A=(r,s)$ depends only on the column number $s$, we can also speak of $\cont_p s$ for any $s\in\Z_{>0}$. 

Let $\la$ be a $p$-strict partition and $i \in I$.
A node $A = (r,s)\in \la$ is {\em $i$-removable} (for $\la$) if one of the following
holds:
\begin{enumerate}
\item[(R1)] $\cont_p A = i$ and
$\la_A:=\la-\{A\}$ is again a $p$-strict partition;
\item[(R2)] the node $B = (r,s+1)$ immediately to the right of $A$
belongs to $\la$,
$\cont_p A = \cont_p B = i=0$,
and both $\la_B = \la - \{B\}$ and
$\la_{A,B} := \la - \{A,B\}$ are $p$-strict partitions.
\end{enumerate}
A node $B = (r,s)\notin\la$ is
{\em $i$-addable} (for $\la$) if one of the following holds:
\begin{enumerate}
\item[(A1)] $\cont_p B = i$ and
$\la^B:=\la\cup\{B\}$ is again an $p$-strict partition;
\item[(A2)]
the node $A = (r,s-1)$
immediately to the left of $B$ does not belong to $\la$,
$\cont_p A = \cont_p B = i=0$, and both
$\la^A = \la \cup \{A\}$ and
$\la^{A, B} := \la \cup\{A,B\}$ are $p$-strict partitions.
\end{enumerate}

Now label all $i$-addable
nodes of $\la$ by $+$ and all $i$-removable nodes of $\la$ by $-$.
The {\em $i$-signature} of
$\la$ is the sequence of pluses and minuses obtained by going along the rim of the Young diagram from bottom left to top right and reading off
all the signs.
The {\em reduced $i$-signature} of $\la$ is obtained
from the $i$-signature
by successively erasing all neighbouring
pairs of the form $+-$. 
Nodes corresponding to  $-$'s in the reduced $i$-signature are
called {\em $i$-normal}. 
The rightmost $i$-normal node is called {\em $i$-good}. Denote
$$\eps_i(\la)  = \sharp\{\text{$i$-normal nodes in $\la$}\}=\sharp\{\text{$-$'s in the reduced $i$-signature of $\la$}\}.
$$
Continuing with the example above, the $0$-addable and $0$-removable nodes are labelled in the diagram:
$$
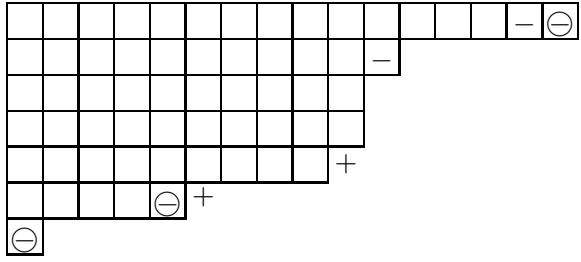
\begin{picture}(320,95)
\put(63,40)
{$\newdimen\hoogte    \hoogte=12pt
\newdimen\breedte   \breedte=14pt
\newdimen\dikte     \dikte=0.5pt
\diagram{
 &  &  &  &  & &  &  &  &  &  &  &  &  & $-$ & $-$ \cr
 &  &  &  &  &  &  &  &  &  & $-$ \cr
 &  &  &  &  &  &  &  &  &   \cr
 &  &  &  &  &  &  &  &  &  \cr
 &  &  &  &  &  &  &  & \cr
 &  &  &  & $-$\cr
$-$ \cr
}
$}
\put(123.4, 14.4){\circle{9}}
\put(69.4, .8){\circle{9}}
\put(271.9, 82.3){\circle{9}}
\put(137,13.5){\makebox(0,0)[b]{$+$}}
\put(191,26){\makebox(0,0)[b]{$+$}}
\end{picture}
\vspace{2 mm}
$$
The $0$-signature of $\la$ is
$-,-,+,+,-,-,-$,
and the reduced $0$-signature is
$-,-,-$.
The nodes corresponding to the $-$'s in the reduced $0$-signature
have been circled in the diagram. The rightmost of them is $0$-good.

Set
$$
\tilde e_i \la =
\left\{
\begin{array}{ll}
\la_A&\hbox{if $A$ is the $i$-good node,}\\
0&\hbox{if $\la$ has no $i$-good nodes},
\end{array}
\right.
$$
The definitions imply that $\tilde e_i \la=0$ or $\tilde e_i\la\in\mathcal{RP}_p(n-1)$ 
if $\la\in \mathcal{RP}_p(n)$. 

\subsection{Crystal graph properties}
We make  $\mathcal{RP}_p:=\bigsqcup_{n\geq 0}\mathcal{RP}_p(n)$ into an $I$-colored directed graph as follows: $\la\stackrel{i}{\rightarrow}\mu$ if and only if $\la= \tilde e_i \mu$.
Kang \cite[Theorem~7.1]{Kang} proves that this graph is isomorphic to $B(\Lambda_0)$, the crystal graph of the basic representation $V(\La_0)$ of the twisted Kac-Moody algebra of type $A_{p-1}^{(2)}$ (interpreted as $B_\infty$ if $p=0$). The Cartan matrix $(a_{ij})_{i,j\in I}$ of this algebra is 
$$
\left(
\begin{matrix}
2 & -2 & 0 & \cdots & 0 & 0 & 0 \\
-1 & 2 & -1 & \cdots & 0 & 0 & 0 \\
0 & -1 & 2 & \cdots & 0 & 0 & 0 \\
 & & & \ddots & & & \\
0 & 0 & 0 & \dots & 2 & -1& 0 \\
0 & 0 & 0 & \dots & -1 & 2& -2 \\
0 & 0 & 0 & \dots & 0 & -1& 2 \\
\end{matrix}
\right)
\quad
\text{if $\ell\geq 2$, and}
$$
\vspace{1mm}
$$
\left(
\begin{matrix}
2 & -4 \\
-1 & 2
\end{matrix}
\right)
\quad \text{if $\ell=1$, and }
$$
\vspace{1mm}
$$
\left(
\begin{matrix}
2 & -2 & 0 &&\\
-1 & 2 & -1 & 0&\\
0 & -1 & 2 & -1& \\
 &0 & -1& 2&\ddots\\
&&&\ddots&\ddots \\
\end{matrix}
\right)
\qquad\text{if $\ell=\infty$.}
$$

In view of Kang's result, we can use some nice properties of crystal graphs established by Stembridge:

\begin{Lemma} \label{TStem} {\rm \cite[Theorem 2.4]{Stem}} 
Let $i,j\in I$ and $i\neq j$. Then
\begin{enumerate}
\item[{\rm (i)}] If $\eps_i(\la)>0$, then $0\leq \eps_j(\tilde e_i\la)-\eps_j(\la)\leq -a_{ji}$.
\item[{\rm (ii)}] If $\eps_i(\la)>0$ and $\eps_j(\tilde e_i\la)=\eps_j(\la)>0$, then $\tilde e_i\tilde e_j\la=\tilde e_j\tilde e_i\la$.
\end{enumerate}
\end{Lemma}

\subsection{Double covers and twisted group algebras}\label{SSDCTGA}
There are two double covers of the symmetric group but the corresponding group algebras over $\FF$ are isomorphic, so it suffices to work with one of them. 
Let $\hat \SSS_n$ be the Schur double cover of the symmetric group $\SSS_n$ in which transpositions lift to involutions. It is known that $\hat \SSS_n$ is generated by elements $z,s_1,\dots,s_{n-1}$ subject only to the relations 
$$
zs_r=s_rz,\ z^2=1,\ s_r^2=1,\ s_rs_{r+1}s_r=s_{r+1}s_rs_{r+1},\ s_rs_t=zs_ts_r\ (|r-t|>1)
$$
for all admissible $r,t$. Then $z$ has order $2$ and generates the center of $\hat \SSS_n$. We have the natural map $\pi:\hat \SSS_n\to \SSS_n$
$$
1\to \langle z\rangle \to \hat\SSS_n\stackrel{\pi}{\to} \SSS_n\to 1
$$
which maps $s_r$ onto the simple transposition $(r,r+1)\in \SSS_n$. The Schur  double cover $\hat \AAA_n$ is $\pi^{-1}(\AAA_n)$. We introduce the {\em twisted group algebras}:
$$
\T_n:=\FF\hat \SSS_n/(z+1),\quad \U_n:=\FF\hat \AAA_n/(z+1).
$$

{\em Spin representations} of $\hat \SSS_n$ and $\hat \AAA_n$ are representations on which $z$ acts non-trivially. The irreducible spin representations are equivalent to the irreducible projective representations of $\SSS_n$ and $\AAA_n$ (at least when $n\neq 6,7$). Moreover, $z$ must act as $-1$ on the irreducible spin representations, so the irreducible spin representations of $\hat \SSS_n$ and $\hat \AAA_n$ are the same as the irreducible representations of the twisted group algebras $\T_n$ and $\U_n$, respectively. From now on we just work with $\T_n$ and $\U_n$. 

We refer the reader to \cite[Section 13.1]{Kbook} for basic facts on these twisted group algebras. In particular, $\T_n$ is generated by the elements $t_1,\dots,t_{n-1}$, where $t_r=s_r+(z+1)$, subject only to the relations 
$$
t_r^2=1,\quad t_rt_{r+1}t_r=t_{r+1}t_rt_{r+1},\quad t_rt_s=-t_st_r\ (|r-s|>1). 
$$
Moreover, $\T_n$ has a natural basis $\{t_g\mid g\in \SSS_n\}$ such that $\U_n=\operatorname{span}(t_g\mid g\in \AAA_n)$. This allows us to introduce a $\Z_2$-grading on $\T_n$ with $(\T_n)_\0=\U_n$ and $(\T_n)_\1=\operatorname{span}(t_g\mid g\in \SSS_n\setminus \AAA_n)$. Thus $\T_n$ becomes a {\em superalgebra}, and we can consider its irreducible {\em supermodules}.

\subsection{\boldmath Supermodules over $\T_n$ and $\U_n$}
Here we review some known results on representation theory of $\T_n$ and $\U_n$ described in detail in \cite[Chapter 22]{Kbook} following \cite{BKproj, BKHC}. It is important that the different approaches of \cite{BKproj} and \cite{BKHC} are ``reconciled'' in \cite{KShch}, where some additional branching results, which will be crucial for us here, are also established. 

First of all, we consider the irreducible {\em supermodules} over $\T_n$. These are labeled by the 
partitions $\la\in \mathcal{RP}_p(n)$. It will be convenient to denote 
\begin{equation}\label{}
\si(m):=
\left\{
\begin{array}{ll}
0 &\hbox{if $m$ is even,}\\
1 &\hbox{if $m$ is odd;}
\end{array}
\right.
\end{equation}
and
\begin{equation}\label{EA}
a(\la):=
\si(n-h_{p'}(\la)).
\end{equation}

The irreducible $\T_n$-supermodule corresponding to $\la\in\mathcal{RP}_p(n)$ will be denoted by $D^\la$, so that
$$
\{D^\la\mid \la\in \mathcal{RP}_p(n)\}
$$
is a complete and irredundant set of irreducible $\T_n$-supermodules up to isomorphism. Moreover, $D^\la$ is of type $\Mtype$ if $a(\la)=0$ and $D^\la$ is of type $\Qtype$ if $a(\la)=1$. Recall the useful fact that $a(\la)$ has the same parity as the number of nodes in $\la$ of non-zero content, see \cite[(22.15)]{Kbook}.

Let $V$ be a $\T_n$-supermodule, $m_1,\dots,m_r\in\Z_{>0}$, and $\mu^1,\dots\mu^r\in\mathcal{RP}_p(n)$. We use the notation $m_1D^{\mu^1}+\dots+m_rD^{\mu^r}\in V$ to indicate that the multiplicity of each $D^{\mu^k}$ as a composition factor of $V$ is at least $m_k$.

\subsection{\boldmath Modules over $\T_n$ and $\U_n$}
Now, we pass from supermodules over $\T_n$ to usual modules  over $\T_n$ and $\U_n$. This is explained in detail in \cite[Section 22.3]{Kbook}. Assume first that $a(\la)=0$. Then $D^\la$ is irreducible as a usual $\T_n$-module. We denote this $\T_n$-module again by $D^\la$. Moreover, $D^\la$ splits into two non-isomorphic irreducible modules on restriction to $\U_n$: $\res^{\T_n}_{\U_n}D^\la=E^\la_+\oplus E^\la_-$. On the other hand, let $a(\la)=1$. Then, considered as a usual module, $D^\la$ splits as two non-isomorphic $\T_n$-modules: $D^\la=D^\la_+\oplus D^\la_-$. Moreover, $E^\la:=\res^{\T_n}_{\U_n}D^\la_+\cong \res^{\T_n}_{\U_n}D^\la_-$ is an irreducible $\U_n$-module. Now,
\begin{align*}
&\{D^\la\mid  \la\in \mathcal{RP}_p(n),\ \text{$a(\la)=0$}\}\,\cup\,
\{D^\la_+,D^\la_-\mid  \la\in \mathcal{RP}_p(n),\ \text{$a(\la)=1$}\}
\end{align*}
is a complete irredundant set of irreducible $\T_n$-modules up to isomorphism, and 
\begin{align*}
&\{E^\la\mid  \la\in \mathcal{RP}_p(n),\ \text{$a(\la)=1$}\}\,\cup\,
\{E^\la_+,E^\la_-\mid  \la\in \mathcal{RP}_p(n),\ \text{$a(\la)=0$}\}
\end{align*}
is a complete  irredundant set of irreducible $\U_n$-modules up to isomorphism. 

We point out that it is usually much more convenient to work with $\T_n$-supermodules, and then `desuperize' at the last moment using the theory described above to obtain results on usual $\T_n$-modules and $\U_n$-modules, cf. \cite[Remark 22.3.17]{Kbook}. 
For future use, we also point out that if $V$ is an irreducible $\T_n$-supermodule and
$W$ is an irreducible constituent of $V$ as a usual $\T_n$-module (or $\hat \SSS_n$-module), then
$$ \frac{\dim V}{\dim W}=2^{a(V)}.$$

\subsection{Weight spaces and superblocks} 
Let $V$ be a $\T_n$-supermodule. We recall the notion of the formal character of $V$ following \cite{BKDurham} and \cite[Section 22.3]{Kbook}. Let $M_1,\dots,M_n$ be the  Jucys-Murphy elements of $\T_n$, cf. \cite[(13.6)]{Kbook}. The main properties of the Jucys-Murphy elements are as follows:

\begin{Theorem} \label{TJM}
We have:
\begin{enumerate}
\item[{\rm (i)}] \cite[Lemma 13.1.1]{Kbook} $M_k^2$ and $M_l^2$ commute for all $1\leq k,l\leq n$;
\item[{\rm (ii)}] \cite[Lemma 22.3.7]{Kbook} if $V$ is a finite-dimensional $\T_n$-supermodule, then for all $1\leq k\leq n$, the eigenvalues of $M_k^2$ on $V$ are of the form $i(i+1)/2$ for some $i\in I$; 
\item[{\rm (iii)}] \cite[Theorem 3.2]{BKDurham} the even center of $\T_n$ is the set of all symmetric polynomials in the 
$M_1^2,\dots,M_n^2$.
\end{enumerate}
\end{Theorem}

For an $n$-tuple $\bi=(i_1,\dots,i_n)\in I^n$, the {\em $\bi$-weight space} of a finite-dimensional $\T_n$-supermodule $V$ is:
$$
V_\bi:=\{v\in V\mid (M_k^2-i_k(i_k+1)/2)^Nv=0\ \text{for $N\gg 0$ and $k=1,\dots,n$}\}.
$$
By Theorem~\ref{TJM}, we have
$
V=\bigoplus_{\bi\in I^n}V_\bi. 
$
 If $V_{\bi}\neq 0$, we say that $\bi$ is a {\em weight} of $V$. 

We denote by
$
\eps_i(V)
$
the maximal non-negative integer $m$ such that $D^\la$ has a non-zero $\bi$-weight space with the last $m$ entries of $\bi$ equal to~$i$.

The superblock theory of $\T_n$ is similar to the usual block theory but uses even central idempotents. 
Denote
$$
\Gamma_n:=\{\ga:I\to \Z_{\geq 0}\mid \sum_{i\in I}\ga(i)=n\}.
$$
Also denote by $\nu_i$ the function from $I$ to $\Z_{\geq 0}$ which maps $i$ to $1$ and $j$ to $0$ for all $j\neq i$. 
For $\ga\in\Ga_n$, we let 
$$
I^\ga:=\{\bi=(i_1,\dots,i_n)\in I^n\mid \nu_{i_1}+\dots+\nu_{i_n}=\ga\}. 
$$
If $V$ is a finite-dimensional $\T_n$-supermodule, then by Theorem~\ref{TJM}(iii),
$$
V[\ga]:=\bigoplus_{\bi\in I^\ga}V_{\bi}
$$
is a $\T_n$-superblock component of $V$, referred to as the {\em $\ga$-superblock component} of $V$, and the decomposition of $V$ into the $\T_n$-superblock components (some of which might be zero) is:
$$
V=\bigoplus_{\ga\in\Ga_n}V[\ga]. 
$$
The {\em $\ga$-superblock} consists of all $\T_n$-supermodules $V$ with $V[\ga]=V$.

Let $\la\in\mathcal{RP}_p(n)$. 
For any $i\in I$ denote by $\ga_i(\la)$ the number of nodes of $\la$ of content $i$. Then we have a function 
$$\ga(\la):=\sum_{i\in I}\ga_i(\la)\nu_i\in\Ga_n.$$ 

\begin{Theorem} \label{TBlocks} {\rm \cite[Theorem 22.3.1(iii)]{Kbook}} 
Let $\la\in \mathcal{RP}_p(n)$ and\, $\ga\in\Ga_n$. Then $D^\la$ is in the $\ga$-superblock of $\T_n$ if and only if $\ga(\la)=\ga$. 
\end{Theorem}


\subsection{Branching rules}

Given a function $\ga:I\to\Z_{\geq 0}$ and $i\in I$ we can consider the function 
$\ga-\nu_i:I\to\Z_{\geq 0}$  if $\ga(i)>0$. 
Now, let $\la\in \mathcal{RP}_p(n)$. Denote 
$$
\res_i D^\la:=\Big(\res^{\T_n}_{\T_{n-1}}D^\la\Big)[\ga(\la)-\nu_i]\qquad(i\in I)
$$
interpreted as zero if $\ga_i(\la)=0$. In other words, 
\begin{equation}\label{ERESI}
\res_i D^\la:=\bigoplus_{\bi\in I^n,\ i_n=i}D^\la_\bi \qquad(i\in I). 
\end{equation}
We have 
$$
\res^{\T_n}_{\T_{n-1}}D^\la=\bigoplus_{i\in I} \res_i D^\la.
$$
Moreover, either $\res_i D^\la$ is zero, or 
$\res_i D^\la$ is self-dual indecomposable, or $\res_i D^\la$ is a direct sum of two self-dual indecomposable supermodules isomorphic to each other and denoted by $e_i D^\la$. If $\res_i D^\la$ is zero or indecomposable we denote $e_iD^\la:=\res_i D^\la$. From now on, for any $\T_n$-supermodule $V$ we will always denote 
$$
\res_{n-j} V:=\res^n_{n-j} V:= \res^{\T_n}_{\T_{n-j}} V.
$$

\begin{Theorem}\label{TBr} {\rm \cite[(22.14), Theorem 22.3.4]{Kbook}, \cite[Theorem A]{KShch}} 
Let $\la\in \mathcal{RP}_p(n)$. There exist $\T_{n-1}$-supermodules $e_i D^\la$ for each $i\in I$, unique up to isomorphism, such that:
\begin{enumerate}
\item[{\rm (i)}] $\res_{n-1} D^\la
$ is isomorphic to 
$$
\left\{
\begin{array}{ll}
e_0 D^\la \oplus 2 e_1 D^\la \oplus \dots \oplus 2 e_\ell D^\la
&\hbox{if $a(\la)=1$,}\\
e_0 D^\la \oplus  e_1 D^\la \oplus \dots \oplus  e_\ell D^\la
&\hbox{if $a(\la)=0$;}\\
\end{array}
\right.
$$
\item[{\rm (ii)}] 
for each $i \in I$,
$e_i D^\la\neq 0$ if and only if $\la$ has an $i$-good node $A$,
in which case $e_i D^\la$ is a self-dual indecomposable 
supermodule with irreducible socle and head isomorphic to $D^{\la_A}$.
\item[{\rm (iii)}] if $\la$ has an $i$-good node $A$, then 
the multiplicity of $D^{\la_A}$ in $e_i D^\la$ is 
$\eps_i(\la)$. Furthermore, $a(D^{\la_A})$ equals $a(D^\la)$ if and only if  $i = 0$;
\item[{\rm (iv)}] if $\mu\in \mathcal{RP}_p(n-1)$ is obtained from $\la$ by removing an $i$-normal node then $D^\mu$ is a composition factor of $e_iD^\la$.
\item[{\rm (v)}] $e_i D^\la$ is irreducible if and only if $\eps_i(\la)=1$. 
\item[{\rm (vi)}] $\res_{n-1}D^\la$ is completely reducible if and only if $\eps_i(\la)=0$ or $1$ for all $i\in I$. 
\item[{\rm (vii)}] $\eps_i(D^\la)=\eps_i(\la)$.
\item[{\rm (viii)}] \cite[Theorem 1.2(ii)]{BKReg} Let $A$ be the lowest removable node of $\la$ such that $\la_A\in\mathcal{RP}_p(n-1)$. Assume that $A$ has content $i$ and that there are $m$ $i$-removable nodes strictly below $A$ in $\la$. Then the multiplicity of $D^{\la_A}$ in $e_i D^\la$ is $m+1$. 
\end{enumerate} 
\end{Theorem}

Finally, one rather special result:

\begin{Lemma} \label{LPhillips3.17} {\rm \cite[Proposition 3.17]{Phillips}} 
Let $p > 3$ and $D,E$ be irreducible $\T_n$-supermodules such that $\res_{n-1}D$ and $\res_{n-1}E$ are both homogeneous with the same unique composition factor. Then $D\cong E$.
\end{Lemma}

\subsection{\boldmath Reduction modulo $p$}
To distinguish between the irreducible modules in characteristic $0$ and $p$ in this section we will use the notation $D^\la_0$ vs. $D^\la_p$. We also distinguish between $I_0=\Z_{\geq 0}$ and $I_p=\{0,1,\dots,\ell\}$. To every $i\in I_0$ we associate $\bar i\in I_p$ via
$\bar i:=\cont_p i$. If $\bi=(i_1,\dots,i_n)\in I_0^n$ then $\bar\bi:=(\bar i_1,\dots,\bar i_n)\in I_p^n$.

Denote reduction modulo $p$ of a finite-dimensional $\T_n$-supermodule $V$ in characteristic zero by $\bar V$. In particular we have $\overline{D_0^\la}$ for any strict partition $\la$ of $n$. 

In fact, let $({\mathbb K},R,\FF)$ be the splitting $p$-modular system which is used to perform reduction modulo $p$. In particular, $\FF=R/(\pi)$ where $(\pi)$ is the maximal ideal of $R$. So we have $\bar V=V_R\otimes_R \FF$ for some $\T_n$-invariant superhomogeneous lattice $V_R$ in $V$.

Recall that $\operatorname{char} \FF\neq 2$ so we may assume that all $i(i+1)/2$ with $i\in I$ belong to the ring of integers $R$. As usual we consider elements of $I_p$ as elements of $\FF$. Then it is easy to see that 
\begin{equation}
i(i+1)/2+(\pi)=\bar i(\bar i+1)/2\qquad (i\in I_0).
\end{equation}

Let again $V$ be an irreducible $\T_n$-supermodule in characteristic zero. When performing its reduction modulo $p$ we can choose a $\T_n$-invariant $R$-lattice $V_R$ of $V$ which respects the weight space decomposition: $V_R=\bigoplus_{\bi\in I_0^n}V_{\bi,R}$, where $V_{\bi,R}=V_R\cap V_\bi$.  Then 
$\overline{ V_\bi}:=V_{\bi,R}\otimes_R\FF\subseteq \bar V_{\bar \bi}$. It follows that for an arbitrary $\bj\in I_p^n$ we have 
\begin{equation}\label{EWTSPRED}
\bar V_\bj=\bigoplus_{\bi\in I_0^n\ \text{such that $\bar\bi=\bj$}} \overline{V_\bi}.
\end{equation}
This implies the following result (cf. the proof of \cite[Lemma~8.1.10]{KShch}): 

\begin{Proposition}\label{PBlockRedModP}
Let $\la$ be a strict partition of $n$ and $D^\la_0$ be the corresponding irreducible $\T_n$-supermodule in characteristic zero. Then all composition factors of the reduction $\overline{D^\la_0}$ modulo $p$ belong to the superblock $\ga$ where $\ga=\sum_{A\in \la}\nu_{\cont_p A}$, where the sum is over all nodes $A$ of $\la$.  
\end{Proposition}

We now use reduction modulo $p$ to deduce some very special results on branching. 

\begin{Lemma} \label{LFactor}
We have:  
\begin{enumerate}
\item[{\rm (i)}] if $p>5$ and $n=p+1$, then $\res_{n-1}D^{(p-1,2)}_p$ has a composition factor $D^\mu$ with $\eps_2(\mu)= 1$;
\item[{\rm (ii)}] if $p>3$ and $n=p+4$, then $\res_{n-1}D^{(p+2,2)}_p$ has a composition factor $D^\mu$ with $\eps_0(\mu)= 2$.
\end{enumerate}
\end{Lemma}
\begin{proof}
We will use the characterization of $\eps_i(\la)$ given in Theorem~\ref{TBr}(vii). 

(i) Let $\ga=3\nu_1+\nu_\ell+2\sum_{i\neq 1,\ell}\nu_i$. Note that $D^{(p-1,2)}_0$ is the only ordinary irreducible in the $\gamma$-superblock, and $D^{(p-1,2)}_p$ is the only $p$-modular irreducible in the $\gamma$-superblock. It follows that $\overline{ D^{(p-1,2)}_0}=mD^{(p-1,2)}_p$ for some multiplicity $m$. So the restriction $\res_{n-1}D^{(p-1,2)}_p$ has the same composition factors as the reduction modulo $p$ of the restriction $\res_{n-1}D^{(p-1,2)}_0=D^{(p-1,1)}_0\oplus D^{(p-2,2)}_0$. Now, note using (\ref{EWTSPRED}) that $\eps_2(\overline{ D^{(p-2,2)}_0})= 1$. 

(ii) Let $\ga=4(\nu_0+\nu_1)+\nu_\ell+2\sum_{i\neq 0,1,\ell}\nu_i$. Note that $D^{(p+2,2)}_0$ is the only ordinary irreducible in the $\gamma$-superblock, and $D^{(p+2,2)}_p$ is the only $p$-modular irreducible in the $\gamma$-superblock. It follows that $\overline{ D^{(p+2,2)}_0}=mD^{(p+2,2)}_p$ for some multiplicity $m$. So the restriction $\res_{n-1}D^{(p+2,2)}_p$ has the same composition factors as the reduction modulo $p$ of the restriction $\res_{n-1}D^{(p+2,2)}_0=D^{(p+2,1)}_0\oplus D^{(p+1,2)}_0$. Now, note using (\ref{EWTSPRED}) that $\eps_0(\overline{ D^{(p+1,2)}_0})= 2$. 
\end{proof}

\section{Basic and second basic modules}

\subsection{Definition, properties, and dimensions}

If the characteristic of the ground field is zero, then the {\em basic}\, supermodule $A_n$ and the {\em second basic}\, supermodule $B_n$ over $\T_n$ are defined to be, respectively,  
$$
A_n:=D^{(n)} \quad \text{and}\quad B_n:=D^{(n-1,1)}.
$$ 

If the ground field has characteristic $p>0$, it follows from the results of \cite{Wales} that reduction modulo $p$ of the characteristic zero basic supermodule has only one composition factor (which could appear with some multiplicity). We define the {\em basic}\, supermodule $A_n$ in characteristic $p$ to be this composition factor. 

Moreover, again by \cite{Wales}, reduction modulo $p$ of the characteristic zero second basic supermodule will always have only one composition factor (with some multiplicity) which is not isomorphic to the basic supermodule---this new composition factor will be referred to as the {\em second basic} supermodule in characteristic $p$ and denoted by $B_n$. 

Thus we have defined the basic supermodule $A_n$ and the second basic supermodule $B_n$ for an arbitrary characteristic. 

When $p>0$, write $n$ in the form
\begin{equation}\label{ENAB}
n=ap+b\qquad(a,b\in \Z,\ 0< b\leq p). 
\end{equation}
Define the functions $\ga^{A_n},\ga^{B_n}\in\Ga_n$ 
by 
\begin{align*}
\ga^{A_n}&:=a(2\nu_0+\dots+2\nu_{\ell-1}+\nu_\ell)+\sum_{s=1}^b\nu_{\cont_p s},\\
\ga^{B_n}&:=a(2\nu_0+\dots+2\nu_{\ell-1}+\nu_\ell)+\sum_{s=1}^{b-1}\nu_{\cont_p s}+\nu_0.
\end{align*}

\begin{Lemma} \label{LABBlocks}
$A_n$ is in the $\ga^{A_n}$-superblock and $B_n$ is in the $\ga^{B_n}$-superblock.  
\end{Lemma}
\begin{proof}
This follows from the definitions of $A_n$ and $B_n$ above in terms of reductions modulo $p$ and Proposition~\ref{PBlockRedModP}. 
\end{proof}

\begin{Theorem} \label{TanW} \cite{Wales} 
We have:
\begin{enumerate}
\item[{\rm (i)}] $\dim A_n= 2^{\lfloor\frac{n-\kappa_n}{2}\rfloor} = 
\left\{
\begin{array}{ll}
2^{\lfloor\frac{n}{2}\rfloor} &\hbox{if $p{\not{|}}n$,}\\
2^{\lfloor\frac{n-1}{2}\rfloor} &\hbox{if $p{{|}}n$;}
\end{array}
\right.
$
\item[{\rm (ii)}] $A_n$ is of type $\Mtype$ if and only if $n$ is odd and 
$~p{\not{|}}n$, or $n$ is even and $p{|}n$.  
\item[{\rm (iii)}] The only possible composition factor of $\res_{n-1}A_n$ is $A_{n-1}$.
\end{enumerate}
\end{Theorem}  

\begin{Theorem} \label{TbnW} {\rm \cite{Wales}} 
We have:
\begin{enumerate}
\item[{\rm (i)}] $\dim B_n= 2^{\lfloor\frac{n-1-\kappa_{n-1}}{2}\rfloor}
                  (n-2-\kappa_n-2\kappa_{n-1})$; equivalently,\\
$\dim B_n  = 
\left\{
\begin{array}{ll}
2^{\lfloor\frac{n-1}{2}\rfloor}(n-2) &\hbox{if $p{\not{|}}n(n-1)$,}\\
2^{\lfloor\frac{n-1}{2}\rfloor}(n-3) &\hbox{if $p{{|}}n$,}\\
2^{\lfloor\frac{n-2}{2}\rfloor}(n-4) &\hbox{if $p{{|}}(n-1)$;}
\end{array}
\right.
$
\item[{\rm (ii)}] $B_n$ is of type $\Mtype$ if and only if $n$ is odd and $p|(n-1)$, or $n$ is even and $p{\not{|}}(n-1)$.  
\item[{\rm (iii)}] The only possible composition factors of $\res_{n-1}B_n$ are $A_{n-1}$ and $B_{n-1}$. 
\end{enumerate}
\end{Theorem}

Finally, we state two results concerning the weights of basic modules. 

\begin{Lemma}\label{LPhillips3.12} {\rm \cite[Corollary 3.12]{Phillips}} 
The only weight appearing in $A_n$ is 
$$(\cont_p0,\cont_p1,\dots,\cont_p(n-1)).$$
\end{Lemma}

\begin{Lemma}\label{LPhillips3.13} {\rm \cite[Lemma 3.13]{Phillips}} 
Let $p>3$ and $D$ be an irreducible $\T_n$-supermodule. Suppose	that	there exist $i,j,k\in I$ (not necessarily distinct) such that every weight $\bi$ appearing in $D$ ends on $ijk$. Then $D$ is basic.
\end{Lemma}

\subsection{Labels}\label{SLabels}
It is important to identify the partitions which label the irreducible modules $A_n$ and $B_n$ in characteristic $p$. Recall the presentation (\ref{ENAB}). Define the partitions $\al_n,\beta_n\in \mathcal{RP}_p(n)$ as follows: 
\begin{align*}
\al_n&:=
\left\{
\begin{array}{ll}
(p^a,b) &\hbox{if $b\neq p$,}\\
(p^{a},p-1,1) &\hbox{if $b=p$;}
\end{array}
\right.
\\
\beta_n&:=
\left\{
\begin{array}{ll}
(n-1,1) &\hbox{if $n<p$,}\\
(p-2,2) &\hbox{if $n=p$,}\\
(p-2,2,1) &\hbox{if $n=p+1$,}\\
(p+1,p^{a-1},b-1) &\hbox{if  $n>p+1$ and $b\neq 1$,}\\
(p+1,p^{a-2},p-1,1) &\hbox{if $n>p+1$ and $b=1$.}
\end{array}
\right.
\end{align*}

For technical reasons we will also need the partition $\ga_n\in\mathcal{RP}_p(n)$ only defined for $n\not \equiv 0,3\pmod{p}$:
$$
\ga_n:=
\left\{
\begin{array}{ll}
(n-2,2) &\hbox{if $n<p$ or $n=p+1$,}\\
(p-1,2,1) &\hbox{if $n=p+2$,}\\
(p+2,p^{a-2},p-1) &\hbox{if $n>p+2$ and $b=1$,}\\
(p+2,p^{a-2},p-1,1) &\hbox{if  $n>p+2$ and $b=2$,}\\
(p+2,p^{a-1},b-2) &\hbox{if $n>p+2$ and $b\neq 1,2,3,p$.}
\end{array}
\right.
$$

Finally, for $p>3$ we define (for $n\not\equiv 1,4\pmod{p}$):
$$
\de_n:=
\left\{
\begin{array}{ll}
\hbox{$(n-3,3)$ or $(n-3,2,1)$} &\hbox{if $n\leq p$,}\\
(p-1,3) &\hbox{if $n=p+2$,}\\
\hbox{$(p-1,3,1)$ or $(p,2,1)$} &\hbox{if $n=p+3$,}\\
\hbox{$(p+2,2,1)$} &\hbox{if $n=p+5>10$,}\\
\hbox{$(p+3,b-3)$ or $(p+2,b-3,1)$} &\hbox{if $a=1$ and $5<b<p$,}\\
\hbox{$(p+2,p-3,1)$ or $(p+2,p-2)$} &\hbox{if $n=2p$,}\\
\hbox{$(p+3,p^{a-2},p-1)$} &\hbox{if $a\geq 2$ and $b=2$,}\\
\hbox{$(p+2,p^{a-1},1)$ or $(p+3,p^{a-2},p-1,1)$} &\hbox{if $a\geq 2$ and $b=3$,}\\

\vspace{2mm}
\hbox{$(p+2,p+1,p^{a-2},2)$} &\hbox{if $a\geq 2$ and $b=5<p$,}\\

\vspace{2 mm}
\hbox{$\begin{array}{ll}(p+3,p^{a-1},b-3)\ \text{or}\\ (p+2,p+1,p^{a-2},b-3)\end{array}$} &\hbox{if $a\geq 2$ and $5<b<p$,}\\
\hbox{$\begin{array}{ll}(p+2,p^{a-1},p-2)\ \text{or}\\ (p+2,p+1,p^{a-2},p-3)\end{array}$} &\hbox{if $a\geq 2$ and $b=p$.}\\
\end{array}
\right.
$$
For $p=3$ we define
$$
\de_n:=(5,3^{a-1},1)\qquad(\text{if $a\geq 2$ and $b=3$}).
$$
(In the cases where $\de_n$ is not unique, this notation is used to refer to any of the two possibilities).

The cases where the formulas above do not produce a partition in $\mathcal{RP}_p(n)$ should be ignored. For example, if $p=3$, there is no $\ga_5$, because the second line of the definition of $\ga_n$ gives $(2,2,1)\notin\mathcal{RP}_3(5)$.

\begin{Theorem} \label{TLabels} 
Let $\la\in\mathcal{RP}_p(n)$. We have:
\begin{enumerate}
\item[{\rm (i)}] $A_n\cong D^{\al_n}$.
\item[{\rm (ii)}] $B_n\cong D^{\beta_n}$.
\item[{\rm (iii)}] If $D^{\al_{n-1}}$ appears in the socle of $\res_{n-1}D^\la$ then $\la=\al_n$ or $\be_n$. 
\item[{\rm (iv)}] If $D^{\beta_{n-1}}$ appears in the socle of $\res_{n-1}D^\la$ then $\la=\be_n$ or $\ga_n$. In particular, $\la$ must be $\beta_n$ if $n\equiv0,3\pmod{p}$.   
\item[{\rm (v)}] If $D^{\ga_{n-1}}$ appears in the socle of $\res_{n-1}D^\la$ then $\la=\ga_n$ or $\de_n$. Conversely, $D^{\ga_{n-1}}$ appears in the socle of 
$\res_{n-1}D^{\de_n}$. 
\end{enumerate} 
\end{Theorem}
\begin{proof}
(i) is proved in \cite[Lemma 22.3.3]{Kbook}. 

(iii), (iv), and (v) come from Theorem~\ref{TBr} by analyzing how good nodes can be added to $\al_{n-1},\be_{n-1}$, and $\ga_{n-1}$, respectively.

(ii) If $n<p$ then the irreducible $\T_n$-supermodules in characteristic $p$ are irreducible reductions modulo $p$ of the irreducible modules in characteristic zero corresponding to the same partition. So the result is clear in this case. We now apply induction on $n$ to prove the result for $n\geq p$. Let $B_n=D^{\beta}$. 
By Theorem~\ref{TbnW}(iii) and the inductive assumption, $\beta$ can be obtained from $\al_{n-1}$ or $\beta_{n-1}$ by adding a good node. 

By (iii), the only partition other than $\al_n$, which can be obtained out of $\al_{n-1}$ by adding a good node is $\beta_n$. Moreover, $\beta_n$ can indeed be obtained out of $\al_{n-1}$ in such a way provided $n\not\equiv 0,1\pmod{p}$. This proves that $\beta=\beta_n$ unless $n\equiv 0,1\pmod{p}$. 

By (iv), the only partition other than $\beta_n$, which can be obtained out of $\beta_{n-1}$ by adding a good node is $\ga_n$. Let $n\equiv 0\pmod{p}$. Then there is no $\ga_n$, and it follows that $\beta=\beta_n$ in this case also. 

Finally, to complete the proof of the theorem, we just have to prove that $\beta=\beta_n$ when $n\equiv 1\pmod{p}$. But we have only two options $\beta=\beta_n$ and $\beta=\ga_n$, and the second one is impossible by Lemma~\ref{LABBlocks}. 
\end{proof}

\subsection{Some branching properties}

\begin{Lemma} \label{LABSoc}
Let $D$ be an irreducible $\T_n$-supermodule. 
\begin{enumerate}
\item[{\rm (i)}] If all composition factors of $\res_{n-1}D$ are isomorphic to $A_{n-1}$, then $D\cong A_n$. 
\item[{\rm (ii)}] If all composition factors of $\res_{n-1}D$ are isomorphic to $A_{n-1}$ or $B_{n-1}$, then $D\cong A_n$ or $D\cong B_n$, with the following exceptions, when the result is indeed false:
\begin{enumerate}
\item[{\rm (a)}] $p>5$, $n=5$, and $D=D^{(3,2)}$;
\item[{\rm (b)}]  $p=5$, $n=6$, and $D=D^{(4,2)}$;
\item[{\rm (c)}] $p=3$, $n=7$, and $D=D^{(5,2)}$.
\end{enumerate}
\item[{\rm (iii)}] Suppose that all composition factors of $\res_mD$ are 
isomorphic to $A_m$ or $B_m$ for some $8\leq  m \leq n$. Then 
$D \cong A_n$ or $D \cong B_n$.
\end{enumerate}
\end{Lemma}
\begin{proof}
(i) is proved in \cite[Lemma 2.4]{KT}. For (ii), if $A_{n-1}$ appears in the socle of $\res_{n-1}D$ then by Theorem~\ref{TLabels}(iii), $D$ is isomorphic to $A_n$ or $B_n$. Thus we may assume that the socle of $D^\la$ is isomorphic to a direct sum of copies of $B_{n-1}=D^{\beta_{n-1}}$. By Theorem~\ref{TLabels}(iv) we just need to rule out the case $D=D^{\ga_n}$. 

When $n<p$ we have $\ga_n=(n-2,2)$, and $D^{(n-3,2)}$ is a composition factor of $\res_{n-1}D^{\ga_n}$, unless $n=5$, when we are in (a), and this is indeed an exception. 

If $n>p$, let $\kappa_{n-1}$ be the partition obtained from $\ga_n$ by removing the bottom removable node. It is easy to see using the explicit definitions of the partitions involved, that $\kappa_{n-1}$ is a restricted $p$-strict partition of $n-1$ different from $\al_{n-1}$ and $\beta_{n-1}$, unless $n=p+1$ or $n=p+4$. Since the bottom removable node is always normal, in the non-exceptional cases we can apply Theorem~\ref{TBr}(iv) to get a composition factor $D^{\kappa_{n-1}}$ in $\res_{n-1}D^{\ga_n}$. 

Now we deal with the exceptional cases $n=p+1$ and $n=p+4$. If $p=3$, then the case $n=p+1$ does not arise since we are always assuming $n\geq 5$. If $n=p+4=7$, we are in the case (c), which is indeed an exception, as for $p=3$ the only irreducible supermodules over $\T_6$ are basic and second basic. 

Similarly, we get the exception (b) for $p=5$, $n=p+1$. All the other cases do not yield exceptions in view of Lemma~\ref{LFactor}. 

To prove (iii), we proceed by induction on $k = n-m$, where the case $k = 0$ is
obvious, and the case $k = 1$ follows from (ii). For the induction step, if 
$U$ is any composition factor of $\res_{n-1}D$, then any composition factor of
$\res_mU$ is isomorphic to $A_m$ or $B_m$. By the induction hypothesis,
$U$ is isomorphic to $A_{n-1}$ or $B_{n-1}$. Hence $D \cong A_n$ or $D \cong B_n$
by (ii).
\end{proof}

In the following two results, which are obtained applying Theorem~\ref{TBr}, $\de_n$ 
means any of the two possibilities for $\de_n$ if $\de_n$ is not uniquely defined. 

\begin{Lemma} \label{GB}
Let $n\geq 6$, and denote $R:=\res_{n-1} D^{\ga_n}$. We have:
\begin{enumerate}
\item[{\rm (i)}] If $n<p$, then $R\cong 2^{\si(n)}(D^{\ga_{n-1}}\oplus D^{\be_{n-1}})$.  
\item[{\rm (ii)}] If $n=p+1$, then $D^{\al_{n-1}}+2D^{\be_{n-1}}\in R$. 
\item[{\rm (iii)}] If $a\geq 2$ and $b=1$, then $2^{\si(n)}(2D^{\be_{n-1}}+D^{\de_{n-1}})\in R$, except for the case $n=7,p=3$, when we have $4D^{\be_{n-1}}\in R$. 
\item[{\rm (iv)}] If $b=2$, then $2^{\si(n+1)}D^{\be_{n-1}}+D^{\ga_{n-1}}\in R$. 
\item[{\rm (v)}] If $a=1$ and $b=4$, then $4D^{\be_{n-1}}\in R$. 
\item[{\rm (vi)}] If $a\geq 2$ and $b=4$, then $2^{\si(n)}(2D^{\be_{n-1}}+D^{\de_{n-1}})\in R$. 
\item[{\rm (vii)}] If $a\geq 1$ and $4<b<p$, then $2^{\si(a+b)}(D^{\be_{n-1}}+D^{\ga_{n-1}})\in R$. 
\end{enumerate}
\end{Lemma}

\vspace{1.5 mm}
\noindent
{\bf Notation. }
Let $\la\in\mathcal{RP}_p(n)$ and $j\in\Z_{>0}$. We denote by $d_j(\la)$ the number of composition factors (counting multiplicities) not isomorphic to $A_{n-j},B_{n-j}$ in $\res^{n}_{n-j}D^\la$. 
\vspace{1.5 mm}

\begin{Lemma} \label{LDeltaBr}
We have $d_1(\de_n)\geq 2$ and $d_2(\de_n)\geq 3$, except possibly in one of the following cases:
\begin{enumerate}
\item[{\rm(i)}] $n=6$, $p>5$, and $\de_n=(3,2,1)$, in which case $\res_{n-1}D^{\de_n}=D^{\ga_{n-1}}$ and $\res_{n-2}D^{\de_n}=2D^{\be_{n-2}}$. 

\item[{\rm (ii)}] $n=7$, $p>3$, and $\de_n=(4,3)$, in which case $\res_{n-1} D^{\de_n}= 2D^{\ga_{n-1}}$, $\res_{n-2} D^{\de_n}=2D^{\be_{n-2}}+2D^{\ga_{n-2}}$ if $p>5$, and $\res_{n-2} D^{\de_n}\ni 4D^{\be_{n-2}}+2D^{\al_{n-2}}$ if $p=5$;

\item[{\rm (iii)}] $n=7$, $p>5$, and $\de_n=(4,2,1)$, in which case 
$\res_{n-1} D^{\de_n}=D^{\ga_{n-1}}+D^{\de_{n-1}}$ and $\res_{n-2} D^{\de_n}=D^{\be_{n-2}}+2D^{\ga_{n-2}}$. 

\item[{\rm (iv)}] $p>3$, $n=p+3$, $\de_n=(p,2,1)$, in which case 
$$\res_{n-1} D^{\de_n}\ni 2D^{\ga_{n-1}}+D^{\al_{n-1}},\quad 
\res_{n-2} D^{\de_n}\ni D^{\al_{n-2}}+2D^{\be_{n-2}}+2D^{\ga_{n-2}}.
$$ 

\item[{\rm (v)}] $p>3$, $n=mp+3$ with $m\geq 2$, $\de_n=(p+2,p^{m-1},1)$, in which case 
$$\res_{n-1} D^{\de_n}\ni 2D^{\ga_{n-1}},\quad 
\res_{n-2} D^{\de_n}\ni 2\cdot 2^{\si(m-1)}D^{\be_{n-2}}+2D^{\ga_{n-2}}.
$$

\item[{\rm (vi)}] $p>5$, $n=p+6$, $\de_n=(p+3,3)$, in which case 
$$\res_{n-1} D^{\de_n}\ni 2D^{\ga_{n-1}},\quad 
\res_{n-2} D^{\de_n}\ni 2D^{\be_{n-2}}+2D^{\ga_{n-2}}.
$$ 

\item[{\rm (vii)}] $p=3$ and $\de_n=(5,3^{a-1},1)$, in which case 
$$
\res_{n-1}D^{\de_n}\ni 2D^{\ga_{n-1}},\quad\res_{n-2}D^{\de_{n}}\ni 2\cdot 2^{\si(a-1)}D^{\be_{n-2}}+2D^{\ga_{n-2}}.
$$
\item[{\rm (viii)}] $p>3$, $n=pm$ for an integer $m\geq 2$, and 
$\de_n=(p+2,p^{m-2},p-2)$, in which case $\res_{n-1} D^{\de_n}= 2^{\si(m)}D^{\ga_{n-1}}$, and 
$$
\res_{n-2}D^{\de_n}\ni 
\left\{
\begin{array}{ll}
2D^{\ga_{n-2}}+2D^{\be_{n-2}} &\hbox{if $p>5$,}\\
2D^{\de_{n-2}}+4D^{\be_{n-2}} &\hbox{if $p=5$ and $n>10$,}\\
4D^{\be_{n-2}} &\hbox{if $p=5$, and $n=10$.}
\end{array}
\right.
$$
\end{enumerate}
\end{Lemma}

\section{Results involving Jantzen-Seitz partitions}

\subsection{JS-partitions}
Let $\la\in\mathcal{RP}_p(n)$. We call $\la$ a {\em JS-partition}, written $\la\in \JS$, if there is $i\in I$ such that $\eps_i(\la)=1$ and $\eps_j(\la)=0$ for all $j\in I\setminus \{i\}$. In this case we also  write $\la\in\JS(i)$ or $D^\la\in \JS(i)$. The notion goes back to \cite{JS,KJS}.

Note that if $\la=(\la_1\geq\la_2\geq\dots\geq \la_h>0)$ is a JS-partition then the bottom removable node $A:=(h,\la_h)$ is the only normal node of $\la$, and in this case we have $\la\in\JS(i)$, where $i=\cont A$. 

\begin{Lemma} \label{Delta-JS}
Let $\de_n$ be one of the explicit partitions defined in \S\ref{SLabels}. 
Then $\de_n\in \JS(i)$ for some $i$ if and only if $p>3$ and one of the following happens:
\begin{enumerate}
\item[{\rm (i)}] $n=6$, $p>5$, and $\de_n=(3,2,1)$; in this case $\de_n\in\JS(0)$ and $a(\la)=1$;
\item[{\rm (ii)}] $n=7$, $p>3$, and $\de_n=(4,3)$; in this case $a(\la)=1$ and  $\de_n\in\JS(2)$;
\item[{\rm (iii)}] $n=mp$ for $m\geq 2$ and $\de_n=(p+2,p^{m-2},p-2)$; in this case $\de_n\in\JS(2)$, $a(\la)=\si(m)$, and 
$$
\res_{n-2}D^{\de_n}\ni 
\left\{
\begin{array}{ll}
2D^{\ga_{n-2}}+2D^{\be_{n-2}} &\hbox{if $p>5$,}\\
2D^{\de_{n-2}}+4D^{\be_{n-2}} &\hbox{if $p=5$ and $n>10$,}\\
4D^{\be_{n-2}} &\hbox{if $p=5$, and $n=10$.}
\end{array}
\right.
$$
\end{enumerate}
\end{Lemma}

\begin{proof}
This is proved by inspection of the formulas for $\de_n$ and applying the definition of the Jantzen-Seitz partitions. 
\end{proof}

Now, we record some combinatorial results of A. Phillips.

\begin{Lemma} \label{LPhillips3.8} {\rm \cite[Lemma 3.8]{Phillips}} 
For  $\la\in\mathcal{RP}_p(n)$ the following are equivalent:
\begin{enumerate}
\item[{\rm (i)}] $\la\in \JS(0)$;
\item[{\rm (ii)}] $\la\in \JS(0)$ and $\tilde e_0\la\in \JS(1)$;
\item[{\rm (iii)}] $\la\in\JS(i)$ and $\tilde e_i\la\in\JS(j)$ for some $i,j\in I$ and exactly one of $i$ and $j$ is equal to $0$.
\end{enumerate}
\end{Lemma}

\begin{Lemma} \label{LPhillips3.14} {\rm \cite[Lemma 3.14]{Phillips}} 
Let  $\la\in\mathcal{RP}_p(n)$. Then: 
\begin{enumerate}
\item[{\rm (i)}] $\la=\al_n$ and $n\equiv1\pmod{p}$ if and only if $\eps_i(\la)=0$ for all $i\neq 0$ and $\tilde e_0(\la)\in \JS(0)$;
\item[{\rm (ii)}] $\la=\al_n$ and $n\not\equiv 0,1,2\pmod{p}$ if and only if $\la\in\JS(i)$ and $\tilde e_i\la\in\JS(j)$ for some $i,j\in I\setminus\{0\}$.
\end{enumerate}
\end{Lemma}

\begin{Lemma} \label{LJS0} {\rm \cite[Lemma 3.7]{Phillips}} 
Let $\la=(l_1^{a_1},\dots,l_m^{a_m})\in\mathcal{RP}_p(n)$ with $l_1>l_2>\dots>l_m>0$. Then $\la\in \JS(0)$ if and only if $l_m=1$ and $\cont_p l_s=\cont_p(l_{s+1}+1)$ for all $s=1,2,\dots,m-1$.
\end{Lemma}

\subsection{Jantzen-Seitz partitions and branching}

\begin{Lemma} \label{L4Cases}
Let $\la\in \JS(i)$ and assume that $D^\la$ is not basic.
Then one of the following happens:
\begin{enumerate}
\item[{\rm (i)}] $i=0$ and $\tilde e_0 \la\in\JS(1)$;
\item[{\rm (ii)}] $i=\ell$, $\eps_{\ell-1}(\tilde e_\ell\la)\geq 2$ and $\eps_j(\tilde e_\ell\la)=0$ for all $j\neq \ell-1$. 
\item[{\rm (iii)}] $i=1$, $\eps_{0}(\tilde e_1\la)\geq 2$ and $\eps_j(\tilde e_1\la)=0$ for all $j\neq 0$. 
\item[{\rm (iv)}] $p>3$, $i\neq 0,\ell$, $\eps_{i-1}(\tilde e_i\la)\geq 1$, $\eps_{i+1}(\tilde e_i\la)=1$ and $\eps_j(\tilde e_i\la)=0$ for all $j\neq i-1,i+1$. Moreover, if in addition, we have $i\neq 1$,  then $\eps_{i-1}(\tilde e_i\la)=1$. 
\end{enumerate}
\end{Lemma}
\begin{proof}
Assume first that $\tilde e_i\la\in\JS(j)$ for some $j$. Then by Lemma~\ref{LPhillips3.14}, exactly one of $i,j$ is $0$. Hence by Lemma~\ref{LPhillips3.8}, we are in (i). 

Now, let  $\tilde e_i\la\not\in\JS$. Then, by Lemma~\ref{TStem}, 
$\eps_j(\tilde e_i \la)>0$ implies that $j=i\pm 1$; moreover $\eps_{i+1}(\tilde e_i\la)\leq 1$, and $\eps_{i-1}(\tilde e_i \la)\leq 1$ if $i\neq 1,\ell$. If $i=\ell$, it now follows that we are in (ii). If $i=1$ we are in (iii) or in (iv). If $i\neq 0,1,\ell$, we are in (iv). 
\end{proof}

\begin{Lemma} \label{LJS1}
Let $\la\in \mathcal{RP}_p(n)$ satisfy Lemma~{\rm \ref{L4Cases}(iv)}. 
Then one of the following occurs:
\begin{enumerate}
\item[{\rm (i)}] $d_2(\la)\geq 4$;
\item[{\rm (ii)}] $a(\la)=0$, $i=1$, and $d_2(\la)\geq 3$; 
\item[{\rm (iii)}] $D^\la\cong B_n$.
\item[{\rm (iv)}] $p> 5$, $n=mp$ for $m\geq 2$, 
$\la=\de_n = (p+2,p^{m-2},p-2)\in\JS(2)$,  and 
$
\res_{n-2}D^{\de_n}\ni 
2D^{\ga_{n-2}}+2D^{\be_{n-2}}.
$
\item[{\rm (v)}] $n=5$, $p>5$, and $\la=(3,2)$.
\item[{\rm (vi)}]  $n=7$, $p>3$, and $\la=(4,3)$.
\end{enumerate}  
\end{Lemma}
\begin{proof}
We may assume that $D^\la$ is not basic.  We may also assume that $D^\la$ is not second basic---otherwise we are in (iii). By Theorem~\ref{TBr} we have
$$
\res_{n-1}D^\la=2^{a(\la)}D^{\tilde e_i\la}.
$$

Assume that $i\neq 1$. Then $i-1\neq 0$ and $a(\tilde e_i\la)+a(\la)=1$, so we have 
$$
\res_{n-2}D^\la=2(D^{\tilde e_{i-1}\tilde e_i\la}+D^{\tilde e_{i+1}\tilde e_i\la}). 
$$
If none of $D^{\tilde e_{i\pm 1}\tilde e_i\la}$ is basic or second basic, we are in (i). 

Suppose that $D^{\tilde e_{i\pm 1}\tilde e_i\la}\cong A_{n-2}$. By Theorem~\ref{TLabels}, we may assume that $\la=\ga_n$. But inspection shows that $\ga_n$ is never JS, unless $n=5$ and $p>5$, in which case, however, $\la\in \JS(1)$. Suppose now that $D^{\tilde e_{i\pm 1}\tilde e_i\la}\cong B_{n-2}$. Then 
we may assume that 
$\la= \de_n$. 
It follows from Lemma~\ref{Delta-JS} that we are in the cases (iv) or (vi). 

Now, let $i=1$. Theorem~\ref{TBr} then gives 
$$
\res_{n-2}D^\la\ni 2^{a(\la)}e_0D^{\tilde e_1\la}+2D^{\tilde e_{2}\tilde e_1\la}.  
$$
If one of $D^{\tilde e_{1\pm 1}\tilde e_1\la}$ is basic or second basic then $\la=\ga_n$ or $\la=\de_n$. If $\la=\ga_n$ then we are in (v). The case $\la=\de_n$ is impossible by Lemma~\ref{Delta-JS}. So we may assume that neither of $D^{\tilde e_{1\pm 1}\tilde e_1\la}$ is basic or second basic. 

If $\eps_0(\tilde e_1\la)\geq 2$, then $D^{\tilde e_{0}\tilde e_1\la}$ appears in $e_0D^{\tilde e_1\la}$ with multiplicity at least $2$, and we are in (i). 
Finally, let $\eps_0(\tilde e_1\la) =\eps_2(\tilde e_1\la)=1$. Then  
$$
\res_{n-2}D^\la=2^{a(\la)}D^{\tilde e_0\tilde e_1\la}+2D^{\tilde e_{2}\tilde e_1\la}.
$$
If $a(\la)=1$, we still get $4$ composition factors, but if $a(\la)=0$, we do get only $3$ composition factors, which is case (ii). 
\end{proof}

\begin{Lemma} \label{LJS2}
Let $p > 3$ and let $\la\in \mathcal{RP}_p(n)$ satisfy Lemma~{\rm \ref{L4Cases}(ii)} or {\rm (iii)}. 
Then one of the following occurs:
\begin{enumerate}
\item[{\rm (i)}] $d_2(\la)\geq 4$;
\item[{\rm (ii)}] $D^\la\cong A_n$.
\item[{\rm (iii)}] $p= 5$, $n=mp$ for $m\geq 2$, 
$\la=\de_n = (p+2,p^{m-2},p-2)$,  and 
$$
\res_{n-2}D^{\de_n}\ni 
\left\{
\begin{array}{ll}
2D^{\de_{n-2}}+4D^{\be_{n-2}} &\hbox{if $n>10$,}\\
4D^{\be_{n-2}} &\hbox{if $n=10$.}
\end{array}
\right.
$$
\end{enumerate}  
\end{Lemma}
\begin{proof}
It follows from the assumption that all weights of $D^\la$ are of the form $(*,i-1,i)$ and that $D^\la$ has a weight of the form $(*,i-1,i-1,i)$. If all weights of $D^\la$ are of the form $(*,i-1,i-1,i)$, then $D^\la$ is basic by Lemma~\ref{LPhillips3.13}. If a weight of the form $(*,i,i-1,i)$ appears in $D^\la$, then so does $(*,i,i,i-1)$ or $(*,i-1,i,i)$ thanks to \cite[Lemma 20.4.1]{Kbook}, which leads to a contradiction. If $(*,j, i-1,i)$ appears with $j\neq i,i-2$, then $(*,i-1,j,i)$ also appears, again leading to a contradiction. So $i=\ell$ and weights of the form $(*,\ell-1,\ell-1,\ell)$ and $(*,\ell-2,\ell-1,\ell)$ appear in $D^\la$. 
In this case $a(\la)+a(\tilde e_\ell\la)=1$, and so Theorem~\ref{TBr} yields a contribution of $4 D^{\tilde e_{\ell-1}\tilde e_\ell\la}$ into $\res_{n-2}D^\la$. So, we are in (i) unless $\tilde e_{\ell-1}\tilde e_\ell\la= \al_{n-2}$ or $\beta_{n-2}$. If $\tilde e_{\ell-1}\tilde e_\ell\la= \al_{n-2}$, then $\la=\be_n$ or $\ga_n$, which never satisfy the assumptions of the lemma. If $\tilde e_{\ell-1}\tilde e_\ell\la= \be_{n-2}$, then we may assume that $\la=\de_n$, which by Lemma~\ref{Delta-JS} leads to the case (iii). 
\end{proof}

Note that if $p=3$ then the cases (ii) and (iii) of Lemma~\ref{L4Cases} are the same.

\begin{Lemma} \label{LJS3}
Let $p=3$ and $\la\in \mathcal{RP}_p(n)$ satisfy Lemma~{\rm \ref{L4Cases}(ii)}. Then 
one of the following occurs:
\begin{enumerate}
\item[{\rm (i)}] $d_2(\la)\geq 4$;
\item[{\rm (ii)}] $\la$ is of the form $(*,5,4,2)$, $a(\la)=0$, in which case $\res_{n-2}D^\la$ has composition factor $D^{(*,5,3,1)}\not\cong A_{n-2},B_{n-2}$ with multiplicity $3$. In particular, $d_2(\la)\geq 3$.
\item[{\rm (iii)}] $D^\la\cong A_n$ or $B_n$.
\end{enumerate}  
\end{Lemma}
\begin{proof}
If $\la$ is neither basic nor second basic, then the assumptions imply that $\la$ has one of the following forms: $(*,5,4,3^a,2)$, $(*,6,4,3^b,2)$, or $(*,5,4,2)$ with $a>0$ and $b\geq 0$. In the first two cases, Theorem~\ref{TBr} gives at least $4$ needed composition factors. So we may assume that we are in (ii). The rest now  follows from Theorem~\ref{TBr}. 
\end{proof}

\subsection{Class \boldmath $\JS(0)$}  
This is the most difficult case since modules $D^\la\in\JS(0)$ tend to branch with very small amount of composition factors.

\begin{Lemma} \label{L220710}
Let $\la\in\mathcal{RP}_p(n)$ and assume that there exist distinct $i,j\in I\setminus\{0\}$ such that $\eps_i(\la)=\eps_j(\la)=1$ and $\eps_k(\la)=0$ for all $k\neq i,j$. Then $\tilde e_i\tilde e_j\la\not\in\JS(0)$.  
\end{Lemma}
\begin{proof}
Assume first that $j\neq 1$. Then by Lemma~\ref{TStem}, we have $\eps_0(\tilde e_j\la)=0$. Now, if $i\neq 1$ then similarly $\eps_0(\tilde e_i\tilde e_j\la)=0$, and $\tilde e_i\tilde e_j\la\not\in\JS(0)$. If $i=1$, we note by Lemma~\ref{LPhillips3.8} that $\sum_k\eps_k(\tilde e_j\la)>1$. So there must exist $k\neq 0,1$ such that $\eps_k(\tilde e_j\la)\geq 1$. Now by Lemma~\ref{TStem}, we have $\eps_k(\tilde e_i\tilde e_j\la)\geq 1$, which shows that $\tilde e_i\tilde e_j\la\not\in\JS(0)$. 

Now assume that $j=1$. Taking into account Lemma~\ref{TStem}, we must have $\eps_i(\tilde e_1\la)=\eps_0(\tilde e_1\la)=1$. By Lemma~\ref{LJS0}, $\tilde e_1\la$ is obtained from $\tilde e_i\tilde e_1\la$ by adding a box of content $i$ to the first row. Now $\la$ must be obtained from $\tilde e_1\la$ by adding a box of residue $1$ to the last row, but then again by Lemma~\ref{LJS0}, we must have $\eps_1(\la)\geq 2$. 
\end{proof}

Our main result on branching of $\JS(0)$-modules is as follows:

\begin{Proposition}\label{JS0}
Let $\la\in\mathcal{RP}_p(n)$ belong to $\la\in \JS(0)$ and $\la\neq \al_n,\be_n$. 
Assume in addition that
\begin{enumerate}
\item[{\rm (i)}] $n>12$ if $p=3$.
\item[{\rm (ii)}] $n>16$ if $p=5$;
\item[{\rm (iii)}] $n>10$ if $p\geq 7$.
\end{enumerate}
Then $d_6(\la)\geq 24$, with three possible exceptions: 
\begin{enumerate}
\item[{\rm (a)}] $p>7$, $\la=(p-3,3,2,1)$, in which case we have
$$
4A_{p-3}+20B_{p-3}+16D^{(p-5,2)}+4D^{(p-6,2,1)}\in\res^{p+3}_{p-3}D^\la. 
$$
\item[{\rm (b)}] $p\geq 7$, $\la=(p+2,p+1,p^a,p-1,1)$ with $a\geq 0$, in which case we have 
$$
4D^{(p+2,p+1,p^a,p-6)}+16D^{(p+2,p^{a+1},p-5)}+4A_{n-6}+20B_{n-6}\in\res_{n-6}D^\la.
$$
\item[{\rm (c)}] $p=5$, $n=18$, and $\la=(7,6,4,1)$, in which case 
$$
20D^{(7,4,1)}+16B_{12}+8A_{12}\in\res_{12}D^\la.
$$
\end{enumerate}
\end{Proposition}

\begin{proof}
We will repeatedly use the notation $\la=(*,l_r^{a_r},l_{r+1}^{a_{r+1}},\dots,l_m^{a_m})$ if we only want to specify the last $m-r+1$ lengths of the parts of $\la$. 

First we consider the case $p=3$. In this case, using Lemma~\ref{LJS0} we see that $\la$ is of the form $(*,2,1)$. Since $n>12$ we could not have $*=\emptyset$, and by Lemma~\ref{LJS0} again, we must have $\la=(*,3^a,2,1)$ with $a>1$ or $\la=(*,4,2,1)$. We could not have $*=\emptyset$ since $\la\neq \al_n,\be_n$, so by Lemma~\ref{LJS0}, we can get more information about $\la$, namely $\la=(*,4,3^a,2,1)$ or $\la=(*,5,4,2,1)$. Since $\la\neq \be_n$ and $n>12$, we conclude that $*\neq \emptyset$ in both cases. 

Now, we get some information on the restriction $\res_{n-6}D^\la$ using Theorem~\ref{TBr}. If $\la=(*,4,3^a,2,1)$, then $2^{a(\la)}D^{(*,4,3^a,1)}\in\res_{n-2}D^\la$. Now, the last node in the last row of length $3$ in $(*,4,3^a,1)$ satisfies the assumptions of Theorem~\ref{TBr}(viii), so we conclude that $2D^{(*,4,3^{a-1},2,1)}\in\res^{n-2}_{n-3} D^{(*,4,3^a,1)}$. Furthermore, the last node in the row of length $4$ in $(*,4,3^a,1)$ is the third normal $0$-node from the bottom. If it is $0$-good, then $3D^{(*,3^{a+1},1)}\in \res^{n-2}_{n-3} D^{(*,4,3^a,1)}$ by Theorem~\ref{TBr}(iii). If it is not good, then the $0$-good node is above it  and $\eps_0(\la)\geq 4$, in which case we get $4D^{(*,4, 3^{a},1)}\in \res^{n-2}_{n-3} D^{(*,4,3^a,1)}$, where by the first $(*,4, 3^{a},1)$ we understand a partition obtained from the second $(*,4, 3^{a},1)$ by removing a box from a row of length greater than $4$. Thus we have 
$$
2^{a(\la)+1}D^{(*,4,3^{a-1},2,1)}+3\cdot 2^{a(\la)}D^{(*,3^{a+1},1)}\in \res_{n-3} D^\la
$$
or 
$$
2^{a(\la)+1}D^{(*,4,3^{a-1},2,1)}+2^{a(\la)} D^{(*,3^{a+1},1)}+4\cdot 2^{a(\la)} D^{(*,4,3^{a},1)}\in \res_{n-3} D^\la.
$$
The second case is much easier so we continue just with the first one. On restriction to $n-4$, we now get 
$$
2^{a(\la)+1}D^{(*,4,3^{a-1},2)}+6\cdot 2^{a(\la)}D^{(*,3^{a},2,1)}\in \res_{n-4} D^\la
$$
Note that $a(\la)+a((*,4,3^{a-1},2))=1$, so we further get
$$
4D^{(*,4,3^{a-1},1)}+6\cdot 2^{a(\la)}D^{(*,3^{a},2)}\in \res_{n-5} D^\la.
$$
Now consider $\res^{n-5}_{n-6}4D^{(*,4,3^{a-1},1)}$. Note that $\eps_0((*,4,3^{a-1},1))\geq 3$, so removal of the $0$-good node yields a contribution of at least $12$ composition factors, none of which is isomorphic to a basic or a second basic module. Finally 
$\res^{n-5}_{n-6}6\cdot 2^{a(\la)}D^{(*,3^{a},2)}$ yields $12D^{(*,3^{a},1)}$, which again cannot be basic or second basic, since here $*$ stands for some parts of length greater than $4$. The restriction $\res^n_{n-6}D^{(*,5,4,2,1)}$ is treated similarly. 

Now, let $p=5$. Using Lemma~\ref{LJS0} and the assumptions $n>16$ and $\la\neq \al_n,\beta_n$, we arrive at the following six possibilities for $\la$: 
\begin{align*}
&(*,5,4,3,2,1), (*,6,4,3,2,1), (*,7,3,2,1),
\\ 
&(*,6,5^a,4,1), (*,7,6,4,1), (*,9,6,4,1),
\end{align*}
with $a\geq 1$ and $*\neq \emptyset$, except possibly in the last two cases. 
Now we use Theorem~\ref{TBr} to show that:
\begin{enumerate}
\item[$\bullet$] $\res_{n-6} D^{(*,5,4,3,2,1)}$ contains $48 D^{(*,5,3,2)}$ or $20 D^{(*,5,3,1)}+4D^{(*,4,3,2)}$ or $20 D^{(*,5,3,1)}+12D^{(*,4,3,2,1)}$.
\item[$\bullet$] $\res_{n-6} D^{(*,6,4,3,2,1)}\ni 4D^{(*,6,4)}+20D^{(*,6,3,1)}$. 
\item[$\bullet$] $\res_{n-6} D^{(*,7,3,2,1)}\ni 20D^{(*,6,1)}+10D^{(*,5,2)}$. 
\item[$\bullet$] $\res_{n-6} D^{(*,6,5^a,4,1)}$ has at least $4$ composition factors of the form $D^{(*,6,5^{a-1},4)}$ and either $20$ composition factors of the form $D^{(*,5^{a},4,1)}$, or $12$ composition factors of the form $D^{(*,5^{a},4,1)}$ and $16$ composition factors of the form $D^{(*,6,5^{a-1},4,1)}$. 
\item[$\bullet$] In the case $*=\emptyset$ we get the exception (c), while in the case $*\neq\emptyset$ we get $\res_{n-6} D^{(*,7,6,4,1)}\ni 20D^{(*,7,4,1)}+4D^{(*,6,5,1)}$.
\item[$\bullet$] $20D^{(*,9,4,1)}+4D^{(*,8,5,1)}\in\res_{n-6} D^{(*,9,6,4,1)}$. 
\end{enumerate}

Finally, let $p\geq 7$. Using Lemma~\ref{LJS0} and the assumptions $n>10$ and $\la\neq \al_n,\beta_n$ we arrive at the following possibilities for $\la$ (with $a\geq 0$): 
\begin{align*}
&(*,4,3,2,1), (*,p-3,3,2,1), (*,p-1,p-2,2,1), (*,p+2,p-2,2,1),\\
&(*,p+2,p+1,p^a,p-1,1), (*,2p-1,p+1,p^a,p-1,1).
\end{align*}
If $\la=(*,4,3,2,1)$ then $*\neq \emptyset$ as $n>10$. In this case we get 
$$
4D^{(*,4)}+20D^{(*,3,1)}\in\res_{n-6}D^\la.
$$
If $\la=(*,p-3,3,2,1)$, we may assume that $p>7$ (otherwise we are in the previous case). If $*=\emptyset$, we are in the exceptional case (a), and Theorem~\ref{TBr} yields the composition factors of the restriction as claimed in the theorem. If $*\neq \emptyset$, we get similar composition factors but with partitions starting with `$*$', and such composition factors are neither basic nor second basic. 

If $\la=(*,p-1,p-2,2,1)$, we have that 
$$12D^{(*,p-1,p-5)}+12D^{(*,p-2,p-4)}\in\res_{n-6}D^\la.$$ 

Let $\la=(*,p+2,p-2,2,1)$. If $*=\emptyset$, then $a(\la)=1$, and using Theorem~\ref{TBr}, we get  $16D^{(p+2,p-5)}+8D^{(p+1,p-5,1)}\in\res_{n-6}D^\la$. Otherwise, we get $16D^{(*,p+2,p-5)}+20D^{(*,p+1,p-4)}\in\res_{n-6}D^\la$. 

If $\la=(*,p+2,p+1,p^a,p-1,1)$, then 
\begin{align*}
4D^{(*,p+2,p+1,p^a,p-6)}+16D^{(*,p+2,p^{a+1},p-5)}+20D^{(*,p+1,p^{a+1},p-4)}&
\\
+4D^{(*,p^{a+2},p-3)}&\in\res_{n-6}D^\la.
\end{align*}
If $*\neq \emptyset$, all of these composition factors are neither basic nor second basic. Otherwise we are in the exceptional case (b). 

The case $\la=(*,2p-1,p+1,p^a,p-1,1)$ is similar to the case $\la=(*,p+2,p+1,p^a,p-1,1)$. 
\end{proof}

We will also need the following result on $\JS(0)$-modules:

\begin{Lemma}\label{JS02}
Let $\la \in {\mathcal {RP}}_p(n)$ for $n \geq 12$. Assume $\la \in \JS(0)$ and
$\la \neq \al_n,\be_n$. Then either 

{\rm (a)} $d_3(\la) \geq 3$, or

{\rm (b)} $d_3(\la) = 2$, $p \geq 5$, and $n = mp+1$ for some $m \geq 2$.  
\end{Lemma}

\begin{proof}
Applying Lemma \ref{L4Cases} to $V := D^{\la}$ we have 
$\res_{n-1}V = U = D^{\mu}$ with $\mu \in \JS(1)$. 
Assume $d_3(V) \leq 2$ so that $d_2(U) \leq 2$. Now we can
apply Lemma \ref{L4Cases} to $\mu \in \JS(1)$ and arrive at one of 
the three cases (ii)--(iv) described in Lemma \ref{L4Cases}. In the
case (ii) (so $p = 3$), the condition $d_2(U) \leq 2$ implies by 
Lemma \ref{LJS3} that $\mu = \al_{n-1}$ or $\be_{n-1}$. In the case (iii)
(and $p > 3$), then since $n \geq 12$ by Lemma \ref{LJS2} either we have 
$\mu = \al_{n-1}$ or we arrive at (b). Similarly, in the case (iv) by Lemma
\ref{LJS1} either we have $\mu = \be_{n-1}$ or we arrive at (b).

Assuming furthermore that (b) does not hold for $V$, we conclude that
$\mu \in \{\al_{n-1},\be_{n-1}\}$. Since $\la \neq \al_n,\be_n$, by Theorem 
\ref{TLabels} we must have $\la = \ga_n$. But then $\la \notin \JS(0)$ by
Lemma \ref{GB}.       
\end{proof}

\section{ The case $\sum\eps_i(\la)=2$}

\subsection{\boldmath The subcase where all $\eps_i(\la)\leq 1$}

\begin{Lemma} \label{LNotBothJS}
Let $\la\in\mathcal{RP}_p(n)$. If there exist $i\neq j$ with $\eps_i(\la)=\eps_j(\la)=1$ 
and $\eps_k(\la)=0$ for all $k\neq i,j$, then at least one of $\tilde e_i\la$, 
$\tilde e_j\la$ is not $\JS$. 
\end{Lemma}
\begin{proof}
Assume that $\tilde e_i\la,\tilde e_j\la\in\JS$. Then by Theorem~\ref{TBr}, we have 
$$
\res_{n-1}D^\la\cong n_1D^{\tilde e_i\la}\oplus n_2D^{\tilde e_j\la}
$$
and 
$$
\res_{n-2}D^\la=n_1m_1D^{\tilde e_j\tilde e_i\la}\oplus n_2m_2 D^{\tilde e_i\tilde e_j\la},
$$
for some $n_1,n_2,m_1,m_2\in\{1,2\}$. 
Moreover, by Lemma~\ref{TStem}, we have $\tilde e_i\tilde e_j\la=\tilde e_j\tilde e_i\la$. It follows that the restrictions $\res_{n-2}D^{\tilde e_i\la}$ and $\res_{n-2}D^{\tilde e_j\la}$ are both homogeneous with the same composition factor $D^{\tilde e_i\tilde e_j\la}$. So, if $p>3$, we get a contradiction with Lemma~\ref{LPhillips3.17}. 

Let $p=3$. Then we may assume that $i=0$ and $j=1$. Note that by the assumption $\eps_0(\la)=\eps_1(\la)=1$, each weight appearing in $D^\la$ ends on $1,0$ or on $0,1$, and both of these occur. After application of $\tilde e_1$ to $D^\la$ only the weights of the from $(*,0,1)$ survive and yield weights of the form $(*,0)$. Since $\tilde e_1\la\in\JS(0)$, we conclude that $\eps_0(\tilde e_1\la)=1$, and so all weights of $D^{\tilde e_1\la}$ are of the form $(*,1,0)$. Similarly all weights of $D^{\tilde e_0\la}$ are of the form $(*,0,1)$. Thus the weights of $D^\la$ are actually of the from $(*,0,1,0)$ and $(*,1,0,1)$. However, by the ``Serre relations'' \cite[Lemma 20.4.2 and Lemma 22.3.8]{Kbook}, the existence of a weight $(*,1,0,1)$ implies the existence of $(*,1,1,0)$ or $(*,0,1,1)$, which now leads to a contradiction. 
\end{proof}

\begin{Lemma} \label{LIJNZ}
Let $\la\in\mathcal{RP}_p(n)\setminus\{\al_n,\be_n,\ga_n,\de_n\}$. Suppose that $\eps_i(\la)=\eps_j(\la)=1$ for some $i\neq j$ in $I\setminus\{0\}$, and $\eps_k(\la)=0$ for all $k\neq i,j$. Then: 
\begin{enumerate}
\item[{\rm (i)}] $\res_{n-1}D^\la\cong 2^{a(\la)}D^{\tilde e_i\la}\oplus2^{a(\la)}D^{\tilde e_j\la}$. Moreover,  $\tilde e_i\la$ and $\tilde e_j\la$ are not both JS, and $\tilde e_i\la, \tilde e_j\la \neq \al_{n-1},\be_{n-1},\ga_{n-1}$. In particular, $d_1(\la)\geq 2$. 
\item[{\rm (ii)}] $d_2(\la)\geq 5$. 
\end{enumerate}
\end{Lemma}
\begin{proof}
By Theorem~\ref{TBr}, we have  
$
\res_{n-1}D^\la\cong 2^{a(\la)}D^{\tilde e_i\la}\oplus2^{a(\la)}D^{\tilde e_j\la}.
$
In view of Lemma~\ref{LNotBothJS}, we now have (i). 

By Lemma~\ref{TStem}, $\eps_i(\tilde e_j\la)>0$ and $\eps_j(\tilde e_i\la)>0$, so 
$$
2^{a(\la)}2^{a(\tilde e_i\la)}D^{\tilde e_j\tilde e_i\la}+2^{a(\la)}2^{a(\tilde e_j\la)}D^{\tilde e_i\tilde e_j\la}=2D^{\tilde e_j\tilde e_i\la}+2D^{\tilde e_i\tilde e_j\la}\in\res_{n-2}D^\la
$$
(it might happen that $\tilde e_i\tilde e_j\la=\tilde e_j\tilde e_i\la$, in which case the above formula is interpreted as $4D^{\tilde e_i\tilde e_j\la}\in \res_{n-2}D^{\la}$). Moreover, since not both $\tilde e_i\la$ and $\tilde e_j\la$ are JS, we may assume without loss of generality that $\tilde e_i\la$ is not JS, i.e. $\sum_k\eps_k(\tilde e_i\la)>1$. Therefore $\eps_j(\tilde e_i\la)\geq 2$ or there exists $k\neq i,j$ with $\eps_k(\tilde e_i\la)>0$. In the first case, we conclude that actually $4D^{\tilde e_j\tilde e_i\la}+2D^{\tilde e_i\tilde e_j\la}\in\res_{n-2}D^\la$, whence $d_2(\la)\geq 6$. In the second case we get $2D^{\tilde e_j\tilde e_i\la}+2D^{\tilde e_i\tilde e_j\la}+2^{a(\la)}D^{\tilde e_k\tilde e_i\la}\in\res_{n-2}D^\la$, so $d_2(\la)\geq 5$. 
\end{proof}

\begin{Lemma} \label{LIJZ}
Let $\la\in\mathcal{RP}_p(n)\setminus\{\al_n,\be_n,\ga_n,\de_n\}$. Suppose that $\eps_i(\la)=\eps_0(\la)=1$ for some $i$ in $I\setminus\{0\}$, and $\eps_k(\la)=0$ for all $k\neq i,0$. Then: 
\begin{enumerate}
\item[{\rm (i)}] $\res_{n-1}D^\la\cong 2^{a(\la)}D^{\tilde e_i\la}\oplus D^{\tilde e_0\la}$. Moreover,  $\tilde e_i\la$ and $\tilde e_0\la$ are not both JS, and $\tilde e_i\la, \tilde e_j\la \neq \al_{n-1},\be_{n-1},\ga_{n-1}$. In particular, $d_1(\la)\geq 2$. 
\item[{\rm (ii)}] $d_2(\la)\geq 3$. 
\end{enumerate}
\end{Lemma}
\begin{proof}
By Theorem~\ref{TBr},  
$
\res_{n-1}D^\la\cong 2^{a(\la)}D^{\tilde e_i\la}\oplus D^{\tilde e_0\la}.
$
In view of Lemma~\ref{LNotBothJS}, we now have (i). 
By Lemma~\ref{TStem}, $\eps_i(\tilde e_0\la)>0$ and $\eps_0(\tilde e_i\la)>0$, so 
$$
2^{a(\la)}D^{\tilde e_0\tilde e_i\la}+2^{a(\tilde e_0\la)}D^{\tilde e_i\tilde e_0\la}=2^{a(\la)}(D^{\tilde e_0\tilde e_i\la}+D^{\tilde e_i\tilde e_0\la})\in\res_{n-2}D^\la.
$$

Moreover, from (i), not both $\tilde e_i\la$ and $\tilde e_0\la$ are JS.
Assume that $\tilde e_i\la\not\in\JS$. Then $\eps_0(\tilde e_i\la)\geq 2$ or there exists $k\neq i,0$ with $\eps_k(\tilde e_i\la)>0$. In the first case, we conclude that actually $2\cdot2^{a(\la)}D^{\tilde e_0\tilde e_i\la}+2^{a(\la)}D^{\tilde e_i\tilde e_0\la}\in\res_{n-2}D^\la$, whence $d_2(\la)\geq 3$. In the second case we get $2^{a(\la)}(D^{\tilde e_0\tilde e_i\la}+D^{\tilde e_i\tilde e_0\la})+2D^{\tilde e_k\tilde e_i\la}\in\res_{n-2}D^\la$, so $d_2(\la)\geq 4$. 
The case $\tilde e_0\la\not\in\JS$ is considered similarly. 
\end{proof}

\begin{Corollary} 
Let $\la\in\mathcal{RP}_p(n)\setminus\{\al_n,\be_n,\ga_n,\de_n\}$, and $i\neq j$ be elements of $I$ such that $\eps_i(\la)\neq 0$, $\eps_j(\la)\neq 0$, and $\eps_k(\la)=0$ for all $k\in I\setminus\{i,j\}$. Then $\res_{n-2}e_i(D^\la)$ or $\res_{n-2}e_j(D^\la)$ is reducible. 
\end{Corollary}
\begin{proof}
If $\eps_i(\la)\geq 2$, then by Lemma~\ref{TStem}, we have $\eps_i(\tilde e_j\la)\geq 2$. Since $D^{\tilde e_j\la}\in e_j(D^\la)$ by Theorem~\ref{TBr}, we conclude that $\res_{n-2} e_j(D^\la)$ is reducible. So we may assume that $\eps_i(\la)=1$ and similarly $\eps_j(\la)=1$. If both $i,j$ are not $0$, we can now use Lemma~\ref{LIJNZ}(i). If one of $i,j$ is $0$ use Lemma~\ref{LIJZ}(i) instead. 
\end{proof}

\subsection{\boldmath The subcase where some $\eps_i(\la)= 2$}
\begin{Lemma} \label{LEpsI2}
Let $\la\in\mathcal{RP}_p(n)\setminus\{\al_n,\be_n,\ga_n,\de_n\}$. Suppose that $\eps_i(\la)=2$ for some $i\in I$, and $\eps_k(\la)=0$ for all $k\neq i$. 
If $\tilde e_i\la\in \JS$, then  $i\neq 0$ and 
$$2^{a(\la)}(2 D^{\tilde e_i\la}+D^\mu)\in \res_{n-1}D^\la,$$ 
where $\tilde e_i\la\neq \al_{n-1},\be_{n-1},\ga_{n-1}$ and $\mu\neq \al_{n-1}$.  
\end{Lemma}
\begin{proof}
First of all, by Lemma~\ref{LPhillips3.14}(i), we have $i\neq 0$. 
By Theorem~\ref{TBr}, 
$$
\res_{n-1}D^\la\cong 2^{a(\la)}e_i(D^\la),
$$
and $2D^{\tilde e_i\la}\in e_i(D^\la)$. Since $\la\neq \al_n,\be_n,\ga_n$, we get $\tilde e_i\la\neq \al_{n-1},\be_{n-1},\ga_{n-1}$. It remains to prove that $e_i(D^\la)$ has another composition factor which is not basic spin. 

The partition $\la$ has two $i$-normal nodes. Denote them by $A$ and $B$, and assume that $A$ is above $B$. Then $A$ is good and $\tilde e_i  \la=\la_A$. Moreover, since the bottom removable node of $\la$ is always normal, we know that $B$ is in the last row. 

Assume first that $\la_B\in\mathcal{RP}_p(n-1)$. In this case $D^{\la_B}\in\res_{n-1}D^\la$ by Theorem~\ref{TBr}(iv). Assume that $\la_B=\al_{n-1}$. Inspecting the formulas for the partitions $\al_{n-1}$ and taking into account the assumption $\la\neq \al_n,\be_n,\ga_n$, we see that $B$ must be of content $0$ which contradicts the assumption $i\neq 0$. 

Assume finally that $\la_B\not\in\mathcal{RP}_p(n-1)$. In this case $\la$ is of the form $\la=(*,k+p,k)$, and $A$ is in the second row from the bottom, i.e. $\la_A=(*,k+p-1,k)$. Since $\la_A\in \JS(i)$, $B$ should be the only normal node of $\la_A$. In particular the node $C$ immediately to the left of $A$ should not be normal in $\la_A$. It follows that $k=(p+1)/2$ and $i=\ell$.

Note that $D^\la$ has a weight of the form 
$$(i_1,\dots,i_{n-3},\ell-1,\ell,\ell)$$ since $\eps_\ell(\la)=2$. 
By \cite[Lemma 20.4.2 and Lemma 22.3.8]{Kbook}, 
$$(i_1,\dots,i_{n-3},\ell,\ell-1,\ell)$$ 
is also a weight of $D^\la$. Therefore $e_{\ell-1}(e_\ell(D^\la))\neq 0$. Since $e_{\ell-1}(D^{\tilde e_\ell\la})=0$, this shows that there is a composition factor $D^\mu$ of $e_\ell(D^\la)$ not isomorphic to $D^{\tilde e_\ell\la}$, and containing the weight $(i_1,\dots,i_{n-3},\ell,\ell-1)$. 

If $\mu=\al_{n-1}$ for all such composition factors, then it follows that all the weights $(i_1,\dots,i_{n-3},\ell,\ell-1)$ are the same and are equal to 
$$(\cont_p0,\cont_p1,\dots,\cont_p(n-1)),$$ 
see Lemma~\ref{LPhillips3.12}. Hence the only weights appearing in $D^\la$ are of the form 
$$(\cont_p0,\cont_p1,\dots,\cont_p(n-3),\ell-1,\ell,\ell)$$ 
or 
$$(\cont_p0,\cont_p1,\dots,\cont_p(n-3),\ell,\ell-1,\ell).$$ 
Hence $D^{\al_{n-3}}$ is the only composition factor of $\res_{n-3}D^\la$. So $D^{\al_{n-2}}$ or $D^{\be_{n-2}}$ are the only modules which appear in the socle of $\res_{n-2}D^\la$. Therefore $D^{\al_{n-1}}$, $D^{\be_{n-1}}$ or $D^{\ga_{n-1}}$ are the only modules which appear in the socle of $\res_{n-1}D^\la$, whence $\la\in\{\al_n,\be_n,\ga_n,\de_n\}$, giving a contradiction. 
\end{proof}

\begin{Lemma} \label{LEpsI2NonJS}
Let $\la\in\mathcal{RP}_p(n)\setminus\{\al_n,\be_n,\ga_n,\de_n\}$. Suppose that $\eps_i(\la)=2$ for some $i\in I$, and $\eps_k(\la)=0$ for all $k\neq i$. Then $d_2(\la)\geq 3$.
\end{Lemma}
\begin{proof}
By Theorem~\ref{TBr}, we have 
$2^{\de_{i,0}}\cdot 2 D^{\tilde e_i^2\la}\in \res_{n-2}D^\la$, so we may assume that $i=0$. Then by Lemma~\ref{LPhillips3.14}, $\tilde e_0\la$ is not $JS$, and hence $\eps_1(\tilde e_0\la)>0$. So $D^{\tilde e_1\tilde e_0\la}$ is also a composition factor of $\res_{n-2}D^\la$. 
\end{proof}

\begin{Lemma}\label{d22}
Let $\la\in\mathcal{RP}_p(n)\setminus\{\al_n,\be_n,\ga_n\}$. If 
$d_2(\la)\leq 2$, then $\la\in \JS(0)$, or 
$\la = \de_n$ and one of the conclusions (i)--(viii) of Lemma 
\ref{LDeltaBr} holds.
\end{Lemma}

\begin{proof}
By Lemma~\ref{LDeltaBr}, we may assume that $\la\neq \de_n$. 
Further, it is clear that we may assume that $\sum_i\eps_i(\la)\leq 2$. If $\la\in \JS(i)$, then it follows from Lemmas~\ref{L4Cases}, \ref{LJS1}, \ref{LJS2}, and \ref{LJS3} that $i=0$. Finally, suppose that $\sum_i\eps_i(\la)=2$. These cases follow from Lemmas~\ref{LIJNZ}, \ref{LIJZ}, and \ref{LEpsI2NonJS}.
\end{proof}

\section{Proof of the Main Theorem}
\subsection{Preliminary remarks}\label{rems}
We denote
\begin{align*}
a_n &:= \dim A_n  = 2^{\lfloor \frac{n-\kappa_n}{2} \rfloor},\\
b_n &:= \dim B_n
 = 2^{\lfloor \frac{n-1-\kappa_{n-1}}{2} \rfloor}(n-2-\kappa_n-2\kappa_{n-1}).
\end{align*}
Define the following non-decreasing functions (of $n$):
\begin{align*}
f(n) &:= 2b_n =  
  2^{\lfloor \frac{n+1-\kappa_{n-1}}{2} \rfloor}(n-2-\kappa_n-2\kappa_{n-1}),
  \\
  \fs(n) &:= \frac{4b_n}{2^{a(\beta_n)}}          
  = 2^{\lfloor \frac{n+2-\kappa_{n-1}}{2} \rfloor}(n-2-\kappa_n-2\kappa_{n-1}).
  \end{align*}
Clearly, $\fs(n) \geq f(n)$. 

We say that an irreducible $\T_n$-supermodule $V$ is 
{\it large}, if it is neither a basic, nor a second basic module. We also 
denote by $d(p,n)$ the smallest dimension of large irreducible 
$\T_n$-supermodules. By Lemma \ref{LABSoc}(iii), the sequence 
$d(p,n)$ is non-decreasing for $n \geq 8$ (and $p$ fixed).
 
\begin{Lemma}\label{trivial}
The Main Theorem  is equivalent to the statement that an irreducible 
$\T_n$-supermodule $V$ satisfying any of the following two conditions

\begin{enumerate}
\item[{\rm (i)}] $\dim V < f(n)$,

\item[{\rm (ii)}] $\dim V < \fs(n)$ and $a(V) = 1$,
\end{enumerate}
is either $A_n$ or $B_n$.
\end{Lemma}

\begin{proof}
Let $W$ be a faithful irreducible $\FF G$-module, where $G = \hat\AAA_n$ or $\hat\SSS_n$,
and consider an irreducible $\T_n$-supermodule $V$ such that $W$ is an 
irreducible constituent of $V$ considered as an $\FF G$-module. If $G = \hat\AAA_n$, then 
$\dim V = 2(\dim W)$, and the bound stated in the Main Theorem for $G=\hat\AAA_n$ is 
precisely $f(n)/2$. Consider the case $G = \hat\SSS_n$. Then 
$\dim V = 2^{a(V)}(\dim W)$, and the bound specified in the Main Theorem for $G=\hat\SSS_n$ 
is $\fs(n)/2$. 

Assume the Main Theorem holds. 
If $\dim V$ satisfies (i), then taking $G = \hat\AAA_n$
we see that $\dim W < f(n)/2$ and so $W$ is a basic or second basic 
representation. If $V$ satisfies (ii), then taking $G = \hat\SSS_n$ we
see that $\dim W < \fs(n)/2$ and so $W$ is again a basic or second basic 
representation. In either case, we can conclude that $V$ is either $A_n$ or 
$B_n$.   

In the other direction, let $\dim W$ satisfy any of the bounds stated in
the Main Theorem. Then $\dim V$ satisfies (i) if $G = \hat\AAA_n$ or 
if $G = \hat\SSS_n$ but $a(V) = 0$, and $\dim V$ satisfies (ii) if
$G = \hat\SSS_n$ and $a(V) = 1$. By our assumption, $V$ is either 
$A_n$ or $B_n$, whence $W$ is a basic or a second basic 
 representation.
\end{proof}

Denote $\pn:= \lfloor (n-\kappa_n)/2\rfloor$. Then
$(n-2)/2 \leq \pn \leq n/2$, and so for $m \leq n$ we have
$$(n-m)/2-1 \leq \pn-\pi_m \leq (n-m)/2+1.$$
In particular, $0 \leq \pn-\pna \leq 1$, and so the sequence $\{\pi_n\}^{\infty}_{n=1}$ is 
non-decreasing; also, $\pna-\pnc \leq 2$.  

\subsection{\boldmath Induction base: $11 \leq n \leq 15$}\label{ind-base}
We will prove the Main Theorem  by induction on $n \geq 11$. 
First, we establish the induction base:

\begin{Lemma} \label{base}
The statement of the Main Theorem  holds true if $12 \leq n \leq 15$, or if $n = 11$ but 
$(n,p,G) \neq (11,3,\HA_{11})$.
\end{Lemma}

\begin{proof}
If $11 \leq n \leq 13$ then one can use \cite{Atlas}, \cite{ModAtl} (and also
decomposition matrices available online at \cite{ModAtl2}) to verify 
the Main Theorem. Also observe that 
\begin{equation}\label{n=13}
  d(p,13) = \left\{ \begin{array}{ll} 3456, & p = 0, 3, 7, \mbox{ or } > 13,\\
                                      2240, & p = 5,\\
                                      1664, & p = 11,\\
                                      2816, & p = 13. \end{array} \right.
\end{equation} 
Now assume that $n = 14$ or $15$. By Lemma \ref{trivial}, it suffices to
show that $\dim V \geq \fs(n)$ for any large 
irreducible $\T_n$-supermodule $V = D^{\la}$. By Lemma \ref{LABSoc}(iii),
$\res_{13}V$ has a large composition factor, and so 
$\dim V \geq d(p,13)$. Direct computation using (\ref{n=13}) shows that 
$d(p,13) \geq \fs(n)$, unless $n = 14$ and $p = 5,11$, or
$n = 15$ and $p = 5,11,13$. To treat these exceptions, we observe that
\begin{equation}\label{n=12}
  d(p,12) = \left\{ \begin{array}{ll} 1408, & p = 11 \mbox{ or } \geq 13,\\
                                      1344, & p = 5, \end{array} \right.
\end{equation} 
in particular, $3d(p,12) > \fs(15)$. So we may assume that 
$d_2(V) \leq 2$, $\dim V < \fs(n)$, and apply Lemma \ref{d22} to $V$.
Moreover, since $d(p,13) > f(14)$, we may also assume $a(V) = 1$ for $n = 14$. 
Furthermore, for $n = 15$ we may assume $V \notin \JS(0)$ as 
otherwise $\dim V \geq 3d(p,12)$ by Lemma \ref{JS02}. 
Now we will rule out the remaining exceptions case-by-case.   

$\bullet$ $(n,p) = (14,11)$. Under this condition, $\ga_{14}$ does not exist, so
either $\la = \de_{14}$ or $V \in \JS(0)$. In the former case, by Lemma
\ref{LDeltaBr} we must have $\de_{14} = (11,2,1)$ and
$$\dim V \geq 2 (\dim D^{\ga_{13}}) + \dim D^{\al_{13}} > 2 \cdot 1664 > 
  2 \cdot 1536 = \fs(14).$$  
In the latter case, $\res_{13}V = D^{\mu}$ with $\mu \in \JS(1)$ and 
$a(D^{\mu}) = a(V) = 1$ by Lemma \ref{L4Cases}. It then follows 
that $\res_{12}V = 2W$ for some faithful irreducible
$\T_{12}$-supermodule $W$. By our assumption,
$$1664 = d(p,13) \leq \dim V = \dim D^{\mu} < \fs(14) = 3072,$$
and $\dim D^{\mu}$ is twice the dimension of some irreducible 
$\HA_{13}$-module. Inspecting \cite{ModAtl2}, we see that $\dim D^{\mu} = 1664$, 
whence $\dim W = 832$. However, $\HA_{12}$ does not have any faithful 
irreducible representation of degree $416$, see \cite{ModAtl}.    
 
$\bullet$ $(n,p) = (14,5)$. Under this condition, $\de_{14}$ does not exist, so
either $\la = \ga_{14}$ or $V \in \JS(0)$. In the former case, by Lemma
\ref{GB} we have
$$\dim V \geq 2 (\dim D^{\be_{13}}) + \dim D^{\de_{13}} > 2(2 \cdot 352 + 1120) > 
  2 \cdot 1536 = \fs(14).$$  
In the latter case, as before we can write 
$\res_{13}V = D^{\mu}$ with $\mu \in \JS(1)$ and $a(D^{\mu}) = a(V) = 1$, 
and $\res_{12}V = 2W$ for some faithful irreducible $\T_{12}$-supermodule $W$. 
By our assumption,
$$2240 = d(p,13) \leq \dim V = \dim D^{\mu} < \fs(14) = 3072.$$
Inspecting \cite{ModAtl2} we see that $\dim D^{\mu}  \in \{2240,2752\}$, so 
$\dim W \in \{ 1120, 1376\}$. 
However, $\HA_{12}$ does not have any faithful 
irreducible representation of degree $560$ or $688$, see \cite{ModAtl}.  

$\bullet$ $(n,p) = (15,5)$. Under this condition $\ga_{15}$ does not exist,
so we need to consider only $\la = \de_{15}$. Now by Lemma \ref{LDeltaBr} we 
have $\la = (7,5,3)$ and 
$$\dim V \geq 2 (\dim D^{\de_{13}}) + 4(\dim D^{\be_{13}}) > 6B_{13} = 4224 >  
  2 \cdot 1536 = \fs(15).$$

$\bullet$ $(n,p) = (15,11)$. Here $\de_{15}$ does not exist,
so we may assume $\la = \ga_{15}$. By Lemmas \ref{LABSoc}(iii) and \ref{GB} we have 
$$\dim V \geq 4 (\dim D^{\be_{14}})+d(p,13) = 4736 >   
  2 \cdot 1664 = \fs(15).$$

$\bullet$ $(n,p) = (15,13)$. By Lemma \ref{LDeltaBr} we may assume
$\la \neq \de_{15}$ and so $\la = \ga_{15}$. Now by Lemma \ref{GB} we have 
$$\dim V \geq \dim D^{\be_{14}} + \dim D^{\ga_{14}} \geq B_{14} + d(p,13) =
  3456 > 2 \cdot 1664 = \fs(15).$$  
\end{proof}

\subsection{\boldmath The third basic representations $D^{\ga_n}$}\label{third}
The following result will be fed into the inductive step in the proof of the Main Theorem:

\begin{Proposition}\label{PGamma}
Let $n \geq 12$ and $V = D^{\ga_n}$. 
Assume in addition that the dimension of any large irreducible $\T_m$-supermodule 
is at least $f(m)$ whenever $12 \leq m \leq n-1$. 
Then $\dim V \geq \fs(n)$. If moreover $V$ 
satisfies the additional condition
\begin{equation}\label{sgam}
  n \geq 15 \mbox{ is odd}, ~p {\!\not{|}}(n-1), \mbox{ and }d_1(V) \geq 2
\end{equation}
then $\dim V \geq \fs(n+1)/2$. 
\end{Proposition}

\begin{proof}
We will proceed by induction on $n \geq 12$ according to the cases in 
Lemma \ref{GB}. 

\smallskip
(i) First we consider the case where $p = 0$ or $p > n$. Then 
$\ga_n = (n-2,2)$. By the dimension formula given in \cite{HH} we have 
$$\dim V = 2^{\lfloor \frac{n-3}{2} \rfloor}(n-1)(n-4).$$
In particular, $\dim V > 4b_n \geq \fs(n)$. Also, $\dim V > \fs(n+1)/2$ if
$n \geq 15$ is odd.

\smallskip
(ii) Next assume that $n = p+1$. 
By Lemma \ref{GB}(ii), 
\begin{equation}\label{b11}
  \dim D^{\ga_n} \geq a_{n-1}+2b_{n-1} = \frac{a_n}{2}+2b_n.
\end{equation}
Since $\fs(n) = 2b_n$ in this case, we get $\dim V > \fs(n)$.

\smallskip
(iii) Assume we are in the case (iii) of Lemma \ref{GB}; 
in particular $n \geq 13$. In this case we have
\begin{equation}\label{b12}
  \frac{\dim D^{\ga_n}}{2^{\sigma(n)}} \geq 2b_{n-1} + \dim D^{\delta_{n-1}} \geq 
  4b_{n-1} =  4b_n.
\end{equation}
It follows that $\dim V \geq 4b_n = 2f(n) \geq \fs(n)$.

\smallskip
(iv) Consider the case (iv) of Lemma \ref{GB}.
If $n = 12$, then $p = 5$, and $\dim V \geq 1344 > 1280 = \fs(12)$.
Assume now that $n \geq 13$ and $a \geq 2$. By Lemma \ref{GB}(iv) and 
(\ref{b12}),
\begin{equation}\label{gam1}
  \dim V \geq 2^{\sigma(n-1)}b_{n-1}+\dim D^{\ga_{n-1}} \geq 
  2^{\sigma(n-1)} \cdot 5b_{n-1} = 
  2^{\lfloor \frac{n-3}{2} \rfloor + \sigma(n-1)}(5n-25).
\end{equation}
On the other hand, 
$$\fs(n) = 2^{\lfloor \frac{n+2}{2}\rfloor}(n-2) = 
    2^{\lfloor \frac{n+1}{2} \rfloor + \sigma(n-1)}(n-2).$$
Hence $\dim(V) \geq \fs(n)$ if $n \geq 17$. If $n = 16$, then $p = 7$. 
In this case, instead of (\ref{b12}) we use the stronger estimate
$$\frac{\dim D^{\ga_{15}}}{2^{\sigma(15)}} \geq 2b_{14} + \dim D^{\delta_{14}} \geq 
  2b_{14} + d(p,13) = 4864,$$
yielding $\dim V \geq 11136 > 7168 = \fs(16)$. If $n = 14$, then $p = 3$,
and $\dim V \geq d(p,13) = 3456 > 3072 = \fs(14)$. The cases 
$n = 13,15$ cannot occur since $n = ap+2$ with $a \geq 2$.
If moreover $V$ satisfies (\ref{sgam}), then since $\res_{n-1}V$ contains an
additional large composition factor in addition to $D^{\ga_{n-1}}$, instead of 
(\ref{gam1}) we now have
$$\begin{array}{ll}\dim V & \geq 2^{\sigma(n-1)}b_{n-1}+\dim D^{\ga_{n-1}}+f(n-1)\\
  & = 2^{(n-3)/2}(7n-35) > 2^{(n+1)/2}(n-1) \geq \fs(n+1)/2.\end{array}$$

Next suppose that $n = p+2 \geq 15$. By Lemma \ref{LABSoc}(iii),  
$\res_{n-2}D^{\ga_{n-1}}$ must contain a large composition factor $Y$,  
and $\dim Y \geq f(n-2) = 2b_{n-2}$ by our assumption. 
It follows by Lemma \ref{GB}(ii) that
$\dim D^{\ga_{n-1}} \geq a_{n-2}+4b_{n-2}$.  
Applying Lemma \ref{GB}(iv), we obtain
\begin{equation}\label{gam2}
  \dim V \geq b_{n-1}+\dim D^{\ga_{n-1}} \geq b_{n-1}+(a_{n-2} + 4b_{n-2}) = 
    2^{\frac{n-3}{2}}(5n-24).
\end{equation}
Since $\fs(n) = 2^{(n+1)/2}\cdot (n-2)$, we are done if $n \geq 16$. If $n = 15$,
then $p = 13$ and by (\ref{n=13}) we have 
$$\dim V \geq b_{14} + \dim D^{\ga_{14}} \geq b_{14} + d(p,13) = 3456 >
   3328 = \fs(15).$$
If $n = 13$, then $p = 11$ and $\dim V \geq d(p,13) = 1664 > 1408 = \fs(13)$ 
by (\ref{n=13}). If moreover $V$ satisfies (\ref{sgam}), then since 
$\res_{n-1}V$ contains an additional large composition factor in addition to 
$D^{\ga_{n-1}}$, instead of (\ref{gam2}) we now have
$$\begin{array}{ll}\dim V & \geq b_{n-1}+\dim D^{\ga_{n-1}}+f(n-1) 
  = 2^{(n-3)/2}(7n-34) \\ 
  & > 2^{(n+1)/2}(n-1) \geq \fs(n+1)/2.\end{array}$$

\smallskip
(v) Now we consider the case $n = p+4$ and $p \geq 11$. Again by Lemma 
\ref{LABSoc}(iii), $\res_{n-1}D^{\ga_n}$ must contain a large composition
factor $X$, and $\dim X \geq f(n-1)$ by our assumption. In fact, since 
$\ga_n$ has exactly one good node (a $1$-good node) with 
two $1$-normal nodes and $a(\ga_n) = 1$, by Theorem \ref{TBr}   
we see that $\res_{n-1}D^{\ga_n} = 2W$, where the $\T_{n-1}$-supermodule
$W$ has $D^{\be_{n-1}}$ as head and socle and $X$ as one of the composition 
factors in between. Thus $X$ has multiplicity at least $2$ in 
$\res_{n-1}D^{\ga_{n-1}}$. Hence by Lemma \ref{GB}(v) we have  
\begin{equation}\label{b41}
  \dim D^{\ga_n} \geq 4b_{n-1} + 2(\dim X) \geq 8b_{n-1} = 2^{\frac{n-3}{2}}(8n-24).
\end{equation}
Since $\fs(n) = 2^{(n+1)/2}(n-2)$ and $\fs(n+1) \leq 2^{(n+3)/2}(n-1)$ 
in this case, we get $\dim V > \max\{\fs(n),\fs(n+1)/2\}$.

\smallskip
(vi) Assume we are in the case (vi) of Lemma \ref{GB}; in particular,
$n \geq 14$. Suppose first that $2|n$. By Theorem \ref{TLabels},
$D^{\ga_{n-2}}$ appears in $\soc(\res_{n-2}D^{\de_{n-1}})$; furthermore, 
$d_1(D^{\de_{n-1}}) \geq 2$ by Lemma \ref{LDeltaBr}. 
Thus $\res_{n-2}D^{\de_{n-1}}$ has at least 
two large composition factors: $D^{\ga_{n-2}}$ and another one, say, $Y$. 
According to (iv), $\dim D^{\ga_{n-2}} \geq \fs(n-2)$. On the other hand,
$\dim Y \geq f(n-2)$ by our assumption. It follows that 
$\dim D^{\de_{n-1}} \geq \fs(n-2)+f(n-2)$. Hence Lemma \ref{GB}(vi) implies 
$$\dim D^{\ga_n} \geq 2b_{n-1} + \dim D^{\delta_{n-1}} \geq 
  2b_{n-1} + \fs(n-2) + f(n-2) =  2^{\frac{n-2}{2}}(5n-18).$$
Since $\fs(n) = 2^{(n+2)/2}(n-2)$, we obtain $\dim V > \fs(n)$. 

Now let $n$ be odd. Then Lemma \ref{GB}(vi) implies that 
\begin{equation}\label{b42}
  \dim D^{\ga_n} \geq 4b_{n-1} + 2(\dim D^{\delta_{n-1}}) \geq 
  8b_{n-1} = 2^{\frac{n-3}{2}}(8n-24). 
\end{equation}
Also, $\fs(n) = 2^{(n+1)/2}(n-2)$ and $\fs(n+1) \leq 2^{(n+3)/2}(n-1)$ 
in this case, so $\dim V > \max\{\fs(n),\fs(n+1)/2\}$.

\smallskip
(vii) Finally, we consider the case (vii) of Lemma \ref{GB}; in particular,
$p \geq 7$ and $n \geq 12$. If $n = 12$, then $p = 7$, and so by 
\cite{ModAtl2} we have 
$\dim V \geq 1408 > 1280 = \fs(12)$. Now we may assume that $n \geq 13$.  

Suppose in addition that $n$ is odd, so that $\sigma(a+b) = 1$.
According to (v) and (vi), $\dim D^{\ga_{n-1}} \geq \fs(n-1) = 4b_{n-1}$.   
Hence by Lemma \ref{GB}(vii) we have
\begin{equation}\label{b43} 
  \dim D^{\ga_n} \geq 2(b_{n-1}+\dim D^{\ga_{n-1}}) \geq 10b_{n-1} = 
  2^{\frac{n-3}{2}}(10n-30). 
\end{equation}
Since $\fs(n) = 2^{(n+1)/2}(n-2)$ and $\fs(n+1) \leq 2^{(n+3)/2}(n-1)$, 
we are done.

Assume now that $n$ is even. If $b = 5$, then 
$\dim D^{\ga_{n-1}} \geq 8b_{n-2}$ by (\ref{b41}) and (\ref{b42}). 
On the other hand, if $b > 5$, then $\dim D^{\ga_{n-1}} \geq 10b_{n-2}$ by 
(\ref{b43}). Thus in either case we have 
$\dim D^{\ga_{n-1}} \geq 8b_{n-2}$. Now Lemma \ref{GB}(vii) implies that 
$$\dim V \geq b_{n-1}+\dim D^{\ga_{n-1}} \geq b_{n-1} + 8b_{n-2} = 
  2^{\frac{n-4}{2}}(10n-38).$$
Since $\fs(n) = 2^{(n+2)/2}(n-2)$, we again have $\dim(V) > \fs(n)$.
\end{proof}

\begin{Proposition}\label{JS4}
Let $n \geq 14$, and let $V = D^\la$ be a large irreducible $\T_n$-supermodule. 
Assume in addition that the dimension of any large irreducible 
$\T_m$-supermodule is at least $f(m)$ whenever $12 \leq m \leq n-1$. 
Then one of the following holds.

{\rm (i)} $d_2(\la) \geq 3$.
 
{\rm (ii)} $\la \in \JS(0)$.

{\rm (iii)} $\la = \ga_n$, $\la \notin \JS$,  and $\dim V \geq \fs(n)$.

{\rm (iv)} $\la = \de_n$, $n \equiv 0,3,6 (\mod~p)$, 
one of the conclusions (iv)--(viii) of 
Lemma \ref{LDeltaBr} holds, and $\dim V \geq \fs(n)$.
\end{Proposition}

\begin{proof}
1) Assume that $\la \notin \JS(0)$ and $d_2(\la) \leq 2$. Then 
we can apply Lemma \ref{d22}. If $\la = \ga_n$, then 
$\la \notin \JS$ (see e.g. Lemma \ref{GB}), and 
$\dim V \geq \fs(n)$ by Proposition \ref{PGamma}. We may now assume that 
$\la = \de_n$, in particular, one of the cases (iv)--(viii) of Lemma 
\ref{LDeltaBr} occurs. By Proposition \ref{PGamma} and our assumptions,
$\dim D^{\ga_m} \geq \fs(m)$ for $m = n-1$ and $m = n-2$.

\smallskip
2) Here we consider the case $n = p+3$ (so that $p \geq 11$). By Lemma 
\ref{LABSoc}(iii), $\res_{n-3}D^{\ga_{n-2}}$ must have some large composition
factor $Z$, and $\dim Z \geq f(n-3) = 2b_{n-3}$ by the assumptions. Applying 
Lemma \ref{GB}(ii), (iv) we get 
\begin{equation}\label{delta1}
  \dim D^{\ga_{n-2}} \geq a_{n-3}+2b_{n-3} + \dim Z,~~~
  \dim D^{\ga_{n-1}} \geq b_{n-2} + \dim D^{\ga_{n-2}}.
\end{equation}
Together with Lemma \ref{LDeltaBr}(iv), this implies
$$\dim V \geq a_{n-1} + 2(\dim D^{\ga_{n-1}}) \geq a_{n-1}+2(a_{n-3}+4b_{n-3}+b_{n-2})
  = 2^{\frac{n-2}{2}}(5n-28).$$
Since $\fs(n) = 2^{(n+2)/2}(n-2)$, we are done if $n \geq 20$. Suppose
that $n \leq 19$, so that $n = p+3 = 16$ or $n = 14$. If $n = 16$, 
then $\dim Z \geq d(p,13) = 2816$, and so (\ref{delta1}) implies
$$\dim D^{\ga_{14}} \geq 4160, ~~~\dim D^{\ga_{15}} \geq 4800.$$
It follows that $\dim V \geq 9728 > 7168 = \fs(16)$. 
If $n = 14$, then $\dim D^{\ga_{13}} \geq d(p,13) = 1664$, so
$$\dim V \geq a_{13} + 2 (\dim D^{\ga_{13}}) = 3392 > 3072 = \fs(14).$$  

\smallskip
3) Next suppose that $n = mp+3$ with $p > 3$ and $m \geq 2$. By  
Lemma \ref{GB}(iii), (iv) we have 
\begin{equation}\label{delta2}
  \dim D^{\ga_{n-2}} \geq 2^{\sigma(n)}(2b_{n-3} + \dim D^{\de_{n-3}}),~~~
  \dim D^{\ga_{n-1}} \geq 2^{\sigma(n)}b_{n-2} + \dim D^{\ga_{n-2}}.
\end{equation}
By our assumptions, $\dim D^{\de_{n-3}} \geq f(n-3) = 2b_{n-3}$. 
Together with Lemma \ref{LDeltaBr}(v), this implies
\begin{equation}\label{delta3}
  \dim V \geq 2(\dim D^{\ga_{n-1}}) \geq 2^{1+\sigma(n)}(b_{n-2}+4b_{n-3})
  = 2^{\sigma(n)+\lfloor \frac{n-2}{2} \rfloor}(5n-30).
\end{equation}
Since $\fs(n) = 2^{\lfloor (n-2)/2\rfloor}(4n-8)$, we are done unless $2|n \leq 20$. 
In the remaining case, $(n,p) = (18,5)$. Then 
$d_1(\de_{15}) \geq 2$ by Lemma \ref{LDeltaBr}, and so
$\dim D^{\de_{15}} \geq 2d(p,13) = 4480$. Thus (\ref{delta2}) implies that
$$\dim D^{\ga_{16}} \geq 7552, ~~~\dim D^{\ga_{17}} \geq 9088,$$
whence $\dim V \geq 18176 > 16384 = \fs(18)$.

\smallskip
4) If $p > 5$ and $n = p+6$, then since $\dim D^{\ga_{n-2}} \geq f(n-2) = 2b_{n-2}$,
by Lemma \ref{LDeltaBr}(vi) we have  
\begin{equation}\label{delta4}
  \dim V \geq 6b_{n-2} = 2^{(n-3)/2}(6n-24) > 2^{(n+1)/2} \cdot (n-2) = \fs(n).
\end{equation}
If $p = 3|n$, then since $\dim D^{\ga_{n-1}} \geq \fs(n-1)$, by Lemma 
\ref{LDeltaBr}(vii) we have  
$$\dim V \geq 2\fs(n-1) \geq 2^{\lfloor \frac{n+1}{2} \rfloor}(2n-6) 
  \geq 2^{\lfloor \frac{n+2}{2} \rfloor} 
  (n-3) = \fs(n).$$
If $5 < p|n$, then using $\dim D^{\ga_{n-2}} \geq \fs(n-2)$ and 
Lemma \ref{LDeltaBr}(viii) we obtain
$$\dim V \geq 2b_{n-2} + 2\fs(n-2) \geq 2^{\lfloor \frac{n-2}{2} \rfloor}(5n-20) > 
  2^{\lfloor \frac{n+2}{2} \rfloor}(n-3) = \fs(n).$$
If $p = 5|n$ and $n$ is odd, then Lemma \ref{LDeltaBr}(viii) and our assumptions 
imply 
$$\dim V \geq 4b_{n-2} + 2f(n-2) = 2^{\frac{n+3}{2}}(n-4) > 
  2^{\frac{n+1}{2}}(n-3) = \fs(n).$$
Finally, assume that $p = 5|n$ and $n \geq 20$ is even. By Lemma \ref{LDeltaBr},
$d_1(\de_{n-2}) \geq 2$, whence $\dim D^{\de_{n-2}} \geq 2f(n-3)$ by our assumptions.
Hence Lemma \ref{LDeltaBr}(viii) yields 
$$\begin{array}{ll}\dim V & \geq 4b_{n-2} + 2(\dim D^{\de_{n-2}}) 
  \geq 4b_{n-2} + 4f(n-3) \\
  & = 2^{n/2}(3n-14) > 2^{(n+2)/2}(n-3) = \fs(n).\end{array}$$
\end{proof}

\subsection{\boldmath The case $V \in \JS$}\label{irred}
\begin{Lemma}\label{bound1}
If $n \geq 23$ and $(n,p) \neq (24,17)$, then $\fs(n) \leq 24f(n-6)$.
\end{Lemma}

\begin{proof}
First assume that $p|(n-7)$. Then $f(n-6) = 2^{\lfloor (n-6)/2 \rfloor}(n-10)$.
In particular, $\fs(n) \leq 24f(n-6)$ if $n \geq 26$.
If $n = 25$, then $p = 3$, 
$\fs(25) = 2^{13}\cdot 21 < 24 \cdot (2^9 \cdot 15) = 24f(19)$.   
If $n = 24$, then $p = 17$. If $n = 23$, then $p > 2$ cannot divide $n-7$.

Next assume that $p{\!\not{|}}(n-7)$. Then 
$f(n-6) \geq 2^{\lfloor (n-5)/2 \rfloor}(n-9)$, and so
$\fs(n) \leq 24f(n-6)$ if $n \geq 23$.
\end{proof}

\begin{Proposition}\label{JSA}
Let $n \geq 16$ and $V\in \JS(0)$ be a large irreducible $\T_n$-supermodule. Assume 
in addition that, if $m := n-6 \geq 12$, then 
the dimension of any large irreducible $\T_m$-supermodule is at least 
$f(m)$. Then $\dim V \geq \fs(n)$.
\end{Proposition}

\begin{proof}
Using the fact that $\ga_n$ is never in $\JS(0)$ (see e.g. Lemma \ref{GB}),
we may assume that $V = D^{\la}$ and $\la \neq \ga_n$.

\smallskip 
(i) First we claim that if $p = 17$ then the dimension of any large 
irreducible $\T_{16}$-supermodule 
$Y = D^{\mu}$ is at least $3d(p,13) = 10368$. 
This is certainly true if $d_j(Y) \geq 3$ for any $j \leq 3$.
Otherwise $d_2(Y) \leq 2$, and so by Lemma \ref{d22} either $\mu \in \JS(0)$, or 
$\mu = \de_{16},\ga_{16}$. In the former case $d_3(Y) \geq 3$ by 
Lemma \ref{JS02}. Also $d_2(\de_{16}) \geq 3$ by Lemma \ref{LDeltaBr}. So we may
assume $\mu = \ga_{16}$. Applying Lemma \ref{GB}(i) three times, we see that
$$\res_{13}Y \cong 2D^{\ga_{13}} + 2b_{13} + 2b_{14}+b_{15}.$$
Since $\dim D^{\ga_{13}} \geq d(p,13)$, we also have $\dim Y > 3d(p,13)$ in this case.

By Lemma \ref{LABSoc}(iii), any large irreducible $\T_{18}$-supermodule 
$X$ has dimension at least $10368$. 

\smallskip
(ii) Now we consider the case $n \geq 23$ and apply Proposition \ref{JS0} to
$\la$. In particular, $d_6(\la) \geq 20$; more precisely,
either $d_6(\la) \geq 24$, or  
$$\dim V \geq 20f(n-6) + 20b_{n-6} + 4a_{n-6} > 30f(n-6).$$
Thus we always have $\dim V \geq 24f(n-6)$. If furthermore $(n,p) \neq (24,17)$,
then the last inequality implies $\dim V \geq \fs(n)$ by Lemma \ref{bound1}. 
Assume now that $(n,p) = (24,17)$. Then by the result of (i) we have
$$\dim V \geq 20 \cdot 10368 > 2^{13} \cdot 22 = \fs(24).$$      
 
\smallskip
(iii) The rest of the proof is to handle the cases $16 \leq n \leq 22$.   

$\bullet$ Consider the case $n = 16,17$. First suppose that $p \neq 5,11$. By 
Lemma \ref{JS02}, $d_3(\la) \geq 3$, hence
$$\dim V \geq 3d(p,13) \geq 8448 > 7680 \geq \fs(n)$$
by (\ref{n=13}). If $(n,p) = (16,5)$, then 
$d_2(\la) \geq 2$ by Lemma \ref{JS02}, whence 
$$\dim V \geq 2d(p,13) \geq 4480 > 3072 = \fs(16)$$
by (\ref{n=13}). On the other hand, the proof of Proposition 
\ref{JS0} shows that if $(n,p) = (16,11)$ then $\la$ can be only $(6,4,3,2,1)$ 
which however does not belong to $\JS(0)$. If $n = 17$ and $p = 5$ or $p = 11$,
then $d_6(\la) \geq 24$ by Proposition \ref{JS0}, whence
$$\dim V \geq 24d(p,11) \geq 24 \cdot 864 > 7680 = \fs(17).$$

$\bullet$ Let $n = 18$. By Proposition \ref{JS0}, $d_6(\la) \geq 24$ if
$p \neq 5$ and $d_6(\la) \geq 20$ if $p = 5$. Now if $p \neq 3$, then  
$$\dim V \geq 20d(p,12) \geq 20 \cdot 1344 > 16384 \geq \fs(18).$$
If $p = 3$, then  
$$\dim V \geq 24d(p,12) = 24 \cdot 640 = 15360 = \fs(18).$$

$\bullet$ Suppose $19 \leq n \leq 21$. By Proposition \ref{JS0}, 
$d_6(\la) \geq 24$ if $(n,p) \neq (20,17)$ and $d_6(\la) \geq 20$ otherwise. 
Now if $(n,p) \neq (20,17)$, then  
$$\dim V \geq 24d(p,13) \geq 24 \cdot 1664 > 38912 \geq \fs(n).$$
If $(n,p) = (20,17)$, then  
$$\dim V \geq 20d(p,13) = 20 \cdot 3456 > 36864 = \fs(20).$$
  
$\bullet$ Finally, let $n = 22$. By Proposition \ref{JS0}, $d_6(\la) \geq 24$ if
$p \neq 19$ and $d_6(\la) \geq 20$ if $p = 19$. 
By the assumptions, the dimension of any large irreducible $\T_{16}$-module
$Y$ is at least $f(16) = 3584$ if $p \neq 5$. We claim that 
$\dim Y > 3584$ also for $p = 5$. (Indeed, by Lemmas \ref{d22}, \ref{JS02},
and \ref{LDeltaBr}, either $d_j(Y) \geq 2$ for some $j \in \{2,3\}$, or
$Y \cong D^{\ga_{16}}$. In the former case, $\dim Y \geq 2d(p,13) = 4480$.
In the latter case, by p. (iii) of the proof of Proposition \ref{PGamma},
$\dim Y \geq 4b_{15} = 6144$.) Now if $p \neq 19$, then  
$$\dim V \geq 24 \cdot 3584 > 81920 \geq \fs(22).$$
If $p = 19$, then by Proposition \ref{JS0} we have  
$$\dim V \geq \min\{20f(16)+20b_{16},24f(16)\} = 24f(16) = 24 \cdot 3584 
  > \fs(22).$$
\end{proof}


\begin{Proposition}\label{JSB}
Let $n\geq 16$ and $V$ be a large irreducible $\T_n$-supermodule. Assume that: 
\begin{enumerate}

\item[{\rm (i)}] $\res_{n-1}V$ is irreducible but $V \notin \JS(0)$;

\item[{\rm (ii)}] the dimension of any large irreducible $\T_m$-supermodule is at 
least $f(m)$ for $12 \leq m \leq n-1$.
\end{enumerate}
Then $a(V) = 0$ and $\dim V \geq f(n)$. 
\end{Proposition}

\begin{proof}
The assumptions in (i) imply that $V \in \JS(i)$ for some $i > 0$ and that 
$a(V) = 0$. By Proposition \ref{JS4} we may assume that 
$d_2(V) \geq 3$ (as otherwise $\dim V \geq \fs(n)$); 
i.e. $\res_{n-2}V$ contains at least three large composition
factors $W_j$, $1 \leq j \leq 3$.  Applying the hypothesis of (ii) to 
$m = n-2$, we get $\dim W_j \geq f(n-2)$ and so  
$\dim V \geq 3f(n-2)$. Assume in addition that $\pna-\pnc \leq 1$.
Then
$$3f(n-2) \geq 2^{\pnc}(6n-36) \geq 
  2^{\pna-1}(6n-36) \geq 2^{\pna+1} \cdot (n-2) \geq f(n),$$
and we are done. 

Next we consider the case $(n,p) = (17,7)$. Then $\res_{13}W_j$ contains a 
large composition factor. Hence, by (\ref{n=13}) 
we have $\dim W_j \geq d(p,13) = 3456$, whence
$\dim V \geq 3 \cdot 3456 > 7680 = f(17)$, and we are done again.   

So we may assume that $\pna-\pnc \geq 2$; equivalently, $n$ is odd
and $p|(n-3)$. Since we have already considered the case $(n,p) = (17,7)$, 
we may assume that $n \geq 21$. It suffices to show that
$\dim W_j \geq f(n)/3$ for $1 \leq j \leq 3$. There are the following 
four possibilities for $W_j$.

$\bullet$ $W_j \cong D^{\ga_{n-2}}$. By Proposition \ref{PGamma} we have
$$\dim W_j \geq \fs(n-2) = 2^{\frac{n-1}{2}}(n-6) > 2^{\frac{n+1}{2}}(n-2)/3 = 
  f(n)/3.$$

$\bullet$ $\res_{n-3}W_{j}$ is reducible but $W_j \not\cong D^{\ga_{n-2}}$. 
Since $W_j$ is large, it must have a large composition factor by Lemma 
\ref{LABSoc}(iii); furthermore, $\res_{n-3}W_j$ can contain neither 
$A_{n-3}$ nor $B_{n-3}$ in its socle. It follows that $d_1(W_j) \geq 2$, and so,
applying the hypothesis of (ii) to $m = n-3$ we get 
$$\dim W_j \geq 2f(n-3) = 2^{\frac{n-1}{2}}(n-6) >  
  2^{\frac{n+1}{2}}(n-2)/3 = f(n)/3.$$

$\bullet$ $W_{j} \in \JS(0)$. Applying Proposition \ref{JS0} to $W_j$ and 
the hypothesis of (ii) to $m = n-8$ we get 
$$\dim W_j \geq 24f(n-8) \geq 24 \cdot 2^{\frac{n-9}{2}}(n-12) \geq 
   2^{\frac{n+1}{2}}(n-2)/3 = f(n)/3.$$
 
$\bullet$ $W_{j} \in \JS(k)$ for some $k > 0$. 
Then $d_2(W_j) \geq 3$ by Proposition 
\ref{JS4} (note that the conclusion (iv) of Proposition \ref{JS4} cannot
hold since $p | (n-3)$). Applying the hypothesis of (ii) to $m = n-4$ we get    
$$\dim W_j \geq 3f(n-4) = 3 \cdot 2^{\frac{n-3}{2}}(n-6) \geq 
  2^{\frac{n+1}{2}}(n-2)/3 =  f(n)/3.$$
  The proposition is proved. 
\end{proof}

\begin{Proposition}\label{JSC}
Let $n\geq 16$ and $V$ be a large irreducible $\T_n$-supermodule. Assume that: 
\begin{enumerate}

\item[{\rm (i)}] $V \in \JS(i)$ for some $i \neq 0$ and $a(V) = 1$;

\item[{\rm (ii)}] for $12 \leq m \leq n-1$, the 
dimension of any large irreducible $\T_m$-supermodule $X$ is at least $f(m)$
if $a(X) = 0$, and at least $\fs(m)$ if $a(X) = 1$.
\end{enumerate}
Then $\dim V \geq \fs(n)$. 
\end{Proposition}

\begin{proof}
1) The assumptions imply that $\res_{n-1}V = 2U$, where $U$ is a 
large irreducible $\T_{n-1}$-supermodule with $a(U) = 0$. 
By Proposition \ref{JS4}, $d_{1}(U) = d_2(V)/2 > 1$ (as otherwise
$\dim V \geq \fs(n)$); in particular, 
$U \notin \JS(0)$. Applying Proposition \ref{JS4} to $U$ we see that either 
$U \cong D^{\ga_{n-1}}$, or $p|(n-1)(n-4)(n-7)$ and $U \cong D^{\de_{n-1}}$, 
or $d_2(U) \geq 3$.

\smallskip
2) Assume we are in the first case: $U \cong D^{\ga_{n-1}}$. Then by Theorem 
\ref{TLabels}, either $V \cong D^{\ga_n}$ or $V \cong D^{\de_n}$. The first 
possibility is ruled out since $V \in \JS$. If the second possibility 
occurs, then Lemma \ref{Delta-JS} implies that $n = mp$ for some $m \geq 2$,
$p > 3$, and $\de_n = (p+2,p^{m-2},p-2)$, which means that $\de_n$ satisfies  
the conclusion (viii) of Lemma \ref{LDeltaBr}. In this case, p. 4) of
the proof of Proposition \ref{JS4} shows that $\dim V \geq \fs(n)$.

\smallskip
3) Consider the second case: $U \cong D^{\de_{n-1}}$ but $d_2(U) \leq 2$. 
Then $\dim U \geq \fs(n-1)$ by Proposition \ref{JS4}.
Now if $p|(n-1)$, then 
$$\dim V \geq 2\fs(n-1) = 2^{\lfloor (n+3)/2 \rfloor}(n-4) > 
  2^{\lfloor (n+1)/2 \rfloor}(n-4) = \fs(n).$$
Likewise, if $5 \leq p|(n-4)$ and $n$ is odd then 
$$\dim V \geq 2\fs(n-1) = 2^{\frac{n+3}{2}}(n-3) > 
  2^{\frac{n+1}{2}}(n-2) = \fs(n).$$
Suppose that $5 \leq p|(n-4)$ and $2|n$; in particular, we are in the case
(v) of Lemma \ref{LDeltaBr}. Then (\ref{delta3}) implies that  
$$\dim V \geq 2^{\frac{n}{2}}(5n-35) > 2^{\frac{n+2}{2}}(n-2) = \fs(n).$$
Suppose that $n = p+7 \geq 16$; in particular, we are in the case
(vi) of Lemma \ref{LDeltaBr}. Then (\ref{delta4}) implies that  
$$\dim V \geq 2^{\frac{n}{2}}(3n-15) > 2^{\frac{n+2}{2}}(n-2) = \fs(n).$$

\smallskip
4) From now on we may assume that $d_2(U) \geq 3$ and so $\res_{n-3}U$ 
contains at least three large composition factors $T_j$, $1 \leq j \leq 3$.  
Applying the hypothesis of (ii) to $m = n-3$, we get $\dim T_j \geq f(n-3)$ 
and so $\dim V \geq 6f(n-3)$. Assume in addition that either  
$n$ is odd, or $2|n \geq 18$ and $p{\!\not{|}}(n-4)$. Then  
$$\dim V \geq 6f(n-3) \geq 6 \cdot 2^{\lfloor \frac{n-2}{2} \rfloor}(n-7)
  \geq 2^{\lfloor \frac{n+2}{2} \rfloor}(n-2) \geq \fs(n).$$
If $n = 16$, then $\dim T_j \geq d(p,13) \geq 1664$ by (\ref{n=13}), whence
$$\dim V \geq 6 \cdot 1664 = 9984 > 7168 \geq \fs(16).$$
If $n \in \{18,20\}$ and $p|(n-4)$, then $(n,p) = (18,7)$, in which case 
$\dim T_j \geq d(p,13) \geq 3456$ by (\ref{n=13}) and so
$$\dim V \geq 6 \cdot 3456 = 20736 > 16384 = \fs(18).$$

\smallskip
5) It remains to consider the case where $n \geq 22$ is even, $p|(n-4)$, and 
$\dim U < \fs(n)/2$. Recall that $U$ is large, $a(U) = 0$, 
$d_1(U) \geq 2$ and $U \not\cong D^{\ga_{n-1}}$. Thus $\res_{n-2}U$ cannot contain
$A_{n-2}$ or $B_{n-2}$ in its socle. Also, since 
$$f(n-2) = 2^{(n-2)/2}(n-4) > \fs(n)/5,$$ 
we have that $\dim U < (5/2)f(n-2)$
and so $d_1(U) \leq 2$ by the hypothesis in (ii) for $m = n-2$. It follows that
$d_1(U) = 2$, i.e. $\res_{n-2}U$ contains exactly two large composition factors
$W_j$, $j = 1,2$. Assume in addition that some $W_j$ has 
$a(W_j) = 1$. By the hypothesis in (ii) for $m = n-2$, in this case we have
$$\dim U \geq f(n-2)+\fs(n-2) = 2^{(n-2)/2}(3n-12) > 2^{n/2}(n-2) \geq \fs(n)/2,$$
and we are done again.

We conclude by Theorem \ref{TBr} that 
$\res_{n-2}U = e_0(U)$ is reducible, with a large irreducible
$\T_{n-2}$-supermodule $W \cong W_1 \cong W_2$ as its socle and head. 
Furthermore, if $p = 3$,
then by the hypothesis in (ii) for $m = n-1$ we have 
$$\dim U \geq f(n-1) = 2^{(n-2)/2}(n-4) = \fs(n)/2.$$  
So we may assume $p > 3$. We will 
distinguish the following three subcases according to Proposition \ref{JS4}
applied to $W$   
(note that $n-2 \equiv 2 (\mod p)$ and so the conclusion (iv) of Proposition  
\ref{JS4} cannot hold) and show that
$\dim W \geq \fs(n)/4$, which contradicts the assumption $\dim U < \fs(n)/2$.

$\bullet$ $d_2(W) \geq 3$. Applying the hypothesis of (ii) to $m = n-4$ we get 
$$\dim W \geq 3f(n-4)= 3 \cdot 2^{(n-4)/2}(n-7) > 2^{(n-2)/2}(n-2) = \fs(n)/4$$
as $n \geq 22$, and so we are done.  

$\bullet$ $W \in \JS(0)$. Since $n \geq 22$, we can apply Proposition 
\ref{JS0} to $W$ and the hypothesis of (ii) to $m = n-8$ to get 
$$\dim W \geq 24f(n-8) \geq
  24 \cdot 2^{(n-8)/2}(n-12) > 2^{(n-2)/2}(n-2) = \fs(n)/4.$$

$\bullet$ $W \cong D^{\ga_{n-2}}$. Recall that $2p|(n-4)$. Hence by
Proposition \ref{PGamma} 
we have 
$$\dim W \geq \fs(n-2) = 2^{n/2}(n-4) > 2^{(n-2)/2}(n-2) = \fs(n)/4.$$
\end{proof}

\subsection{Inductive step of the proof of the Main Theorem}
As a consequence of the results proved in \S\S\ref{rems} -- \ref{irred} we obtain
the following:

\begin{Corollary}\label{induction}
For the induction step of the proof of the Main Theorem, it suffices to prove 
that, if $V = D^{\la}$ is any irreducible 
$\T_n$-supermodule satisfying all the following conditions

\begin{enumerate}

\item[{\rm (i)}] $n \geq 16$, $\la \neq \al_n,\beta_n,\ga_n$;

\item[{\rm (ii)}] $V \notin \JS$, $d_1(V) \geq 2$, $d_2(V) \geq 3$,
and all the simple summands of the head
and the socle of $\res_{n-1}V$ are large
\end{enumerate}
then $\dim V \geq f(n)$, and, furthermore, $\dim V \geq \fs(n)$ when $a(V) = 1$.
\end{Corollary}

\begin{proof}
By the induction hypothesis, the dimension of any 
irreducible $\T_m$-supermodule $X$ is at least $f(m)$ if
$a(X) = 0$ and at least $\fs(m)$ if $a(X) = 1$ for $12 \leq m \leq n-1$.
By Lemma \ref{base} and Propositions \ref{PGamma}, \ref{JSA} 
we may now assume that $n \geq 16$,
$\la \neq \al_n,\beta_n,\ga_n$ and $V \notin \JS(0)$. Now, if $\res_{n-1}V$ is 
irreducible, then $V \in \JS(i)$ for some $i > 0$ and $a(V) = 0$, in which case
we also have $\dim V \geq f(n)$ by Proposition \ref{JSB}. 
The case $V \in \JS(i)$ with $a(V) = 1$ is treated in Proposition \ref{JSC}.
So we may assume that $V \notin \JS$. Since $\la \neq \al_n,\beta_n,\ga_n$, 
$\res_{n-1}V$ cannot contain $A_{n-1}$ or $B_{n-1}$ in the socle or in the head. 
It now follows that $d_1(V) \geq 2$. Also, if $d_2(V) \leq 2$, then
we may assume  
$\dim V \geq \fs(n)$ by Proposition \ref{JS4}.
\end{proof}

Now we will complete the induction step of the proof of the Main Theorem. 
Arguing by contradiction, we will assume that the irreducible 
$\T_n$-supermodule $V$ satisfies 
the conditions listed in Corollary \ref{induction}, but 
$$\dim V < \left\{ \begin{array}{rl}f(n), & \mbox{ if }a(V) = 0,\\ 
                            \fs(n), & \mbox{ if }a(V) = 1.\end{array} \right.$$
The condition $d_1(V) \geq 2$ implies that $\res_{n-1}V$ contains at least 
two large composition factors $U_j$, $j = 1,2$, and $\dim U_j \geq f(n-1)$
by the induction hypothesis, whence $\dim V \geq 2f(n-1)$. Similarly,
the condition $d_2(V) \geq 3$ implies that $\dim V \geq 3f(n-2)$.

We distinguish between the following three cases.

\subsubsection{{\bf Case I:} $\pna-\pnc = 2$} This case happens precisely
when $n$ is odd and $p|(n-3)$, whence 
$$\fs(n) = f(n) = 2^{\frac{n+1}{2}}(n-2-\ka_n),~~~
  f(n-1) = 2^{\frac{n-1}{2}}(n-3) = \frac{\fs(n-1)}{2}.$$
In particular, if $p = 3$ then $\fs(n) = 2f(n-1) \leq \dim V$.
So we may assume $p > 3$. Then 
$$\dim V -2f(n-1) < f(n)-2f(n-1) = 2^{(n+1)/2} = 2a_{n-1} < b_{n-1} < f(n-1).$$
It follows that $d_1(V) = 2$, and aside from $U_1$, $U_2$, 
$\res_{n-1}V$ can have at most one more composition factor which is then 
isomorphic to $A_{n-1}$. Also, if $a(U_j) = 1$ for some $j$, 
then by the induction hypothesis, $\dim U_j \geq \fs(n-1) = 2f(n-1)$, and so
we would have $\dim V \geq 3f(n-1) > f(n)$. Thus $a(U_j) = 0$ for $j = 1,2$. 

Suppose that $a(V) = 0$. The above conditions on $\res_{n-1}V$ imply by
Theorem \ref{TBr} that $\res_{n-1}V = e_0(V)$ has socle and head both isomorphic
to $U \cong U_1 \cong U_2$. Since $d_2(V) \geq 3$
(and all composition factors of $\res_{n-2}A_{n-1}$ are isomorphic to $A_{n-2}$),
we see that $d_1(U) \geq 2$; in particular, $U \notin \JS(0)$. Also, 
$\dim U \leq (\dim V)/2 < \fs(n-1)$. Hence Proposition \ref{JS4} 
applied to $U$ yields $d_2(U) \geq 3$. It follows that 
$$\dim V \geq 2(\dim U) \geq 6f(n-3) = 2^{\frac{n-3}{2}}(6n-36) > 
  2^{\frac{n+1}{2}}(n-2) = f(n).$$ 

Next suppose that $a(V) = 1$. Then the above conditions on $\res_{n-1}V$ imply by
Theorem \ref{TBr} that $\res_{n-1}V = 2e_i(V) = 2U$ with
$U \cong U_1 \cong U_2$ and $i > 0$. 
Since $d_2(V) \geq 3$ we see that $d_1(U) \geq 2$ and so $U \notin \JS(0)$. Also, 
$\dim U \leq (\dim V)/2 < \fs(n-1)$. Hence Proposition \ref{JS4} 
applied to $U$ again yields $d_2(U) \geq 3$ and 
$\dim V \geq 6f(n-3) > f(n)$. In either case we have reached a contradiction.

\subsubsection{{\bf Case II:} $\pna-\pnb = 0$} 
This case happens precisely when either $p|(n-1)$, or
$p{\not{|}}(n-1)(n-2)$ and $2|n$. In the former case,
$$\fs(n) = 2^{\lfloor \frac{n+1}{2}\rfloor}(n-4) \leq 
  2^{1+\lfloor \frac{n}{2}\rfloor}(n-4) = 2f(n-1) \leq \dim V$$
a contradiction. Likewise, in the latter case, 
$$f(n) = 2^{\frac{n}{2}}(n-2-\ka_n) \leq 
  2^{1+\frac{n}{2}}(n-3) = 2f(n-1) \leq \dim V.$$ 
If in addition $p|n$, then 
$$\fs(n) = 2^{1+\frac{n}{2}}(n-3) = 2f(n-1) \leq \dim V.$$
Hence we may assume that $p{\!\not{|}}n(n-1)(n-2)$, $2|n$, and $a(V) = 1$.
In this case
$$\dim V -2f(n-1) < \fs(n)-2f(n-1) = 2^{(n+2)/2} = 4a_{n-1} < b_{n-1} < f(n-1).$$
It follows that $d_1(V) = 2$, and aside from $U_1$, $U_2$, all other
composition factors of $\res_{n-1}V$ (if any) must be isomorphic to $A_{n-1}$.

Suppose in addition that $e_i(V) \neq 0$ for some $i > 0$. 
Then we may assume that $U_1$ is in $\soc(e_i(V))$. 
As $a(V) = 1$, $2e_i(V)$ is a direct summand of $\res_{n-1}V$.
In particular, if there is some $k \neq i$ such that $e_k(V) \neq 0$, then
$\soc(e_k(V))$ must be $A_{n-1}$, contrary to our hypotheses. Thus 
$\res_{n-1}V = 2e_i(V)$ in this case. Now $e_i(V)$ has a composition factor
$U_1$ with multiplicity one and all other composition factors (if any)
are isomorphic to $A_{n-1}$. By our hypotheses, $\soc(e_i(V)) = U_1$. It follows
that $\varepsilon_i(\la) = 1$, and so $e_i(V) = U_1$ is irreducible by
Theorem \ref{TBr}(v). Thus $V \in \JS(i)$, a contradiction.

We have shown that $\res_{n-1}V = e_0(V)$, with
$U := U_1 = \soc(e_0(V)) \cong \head(e_0(V)) = U_2$,
$\varepsilon_0(\la) = 2$, and $a(U) = a(V) = 1$.
Now $d_1(U) = d_2(V)/2 > 1$; in particular, $U \notin \JS(0)$. Thus we can
apply Proposition \ref{JS4} and distinguish the following subcases.

\smallskip
(a) {\it Suppose $d_2(U) \geq 3$ and $p{\!\not{|}}(n-4)$.} Then 
$$\dim V \geq 2(\dim U) \geq 6f(n-3) \geq 2^{(n-2)/2}(6n-36) 
  > 2^{(n+2)/2}(n-2) = \fs(n).$$  

\smallskip
(b) {\it Suppose $p|(n-4)$ and $U \not\cong D^{\ga_{n-1}}$.} Recall that 
$d_1(U) \geq 2$. If $d_1(U) \geq 3$, or if some large composition
factor $X$ of $\res_{n-2}U$ has $a(X) = 1$, then since 
$\fs(n-2) = 2f(n-2)$, the induction hypothesis implies 
$$\dim V \geq 2(\dim U) \geq 6f(n-2) \geq 2^{(n-2)/2}(6n-24) 
  > 2^{(n+2)/2}(n-2) = \fs(n).$$
Thus $d_1(U) = 2$ and every large composition factor $W$ of $\res_{n-2}U$ has 
$a(W) = 0$. Moreover, the socle and head of $\res_{n-2}U$ can contain neither
$A_{n-2}$ nor $B_{n-2}$. It follows by Theorem \ref{TBr} that 
$\res_{n-2}U = 2e_i(U) = 2W$ for some $i > 0$ and some 
irreducible $\T_{n-2}$-supermodule $W$. In particular, $U \in \JS(i)$.
We have shown that $\varepsilon_k(\la) = 2\delta_{k,0}$ and 
$\tilde{e}_0\la = U \in \JS$. Furthermore, $\la \neq \ga_n$ by our assumption.
Hence, by Lemma \ref{LEpsI2} we must have $\la = \de_n$. But in this case
$D^{\ga_{n-1}}$ appears in the socle of $\res_{n-1}V$ by Theorem \ref{TLabels}(v).
Thus $U \cong D^{\ga_{n-1}}$, contrary to our assumption.

\smallskip
(c) {\it Suppose $p{\!\not{|}}(n-4)$, $d_2(U) \leq 2$ and 
$U \not\cong D^{\ga_{n-1}}$.} Since $p{\!\not{|}}(n-1)$ and 
$U\notin \JS(0)$, by Proposition \ref{JS4} this can happen only when 
$n = p+7$ (so that $p \geq 11$), and $U = D^{\de_{n-1}}$ as specified in 
Lemma \ref{LDeltaBr}(vi). Applying Lemma \ref{LDeltaBr}(vi) and 
Proposition \ref{PGamma}, we obtain   
$$\dim V \geq 2(\dim U) \geq 4\fs(n-2) \geq 2^{n/2}(4n-16) 
  > 2^{(n+2)/2}(n-2) = \fs(n).$$

\smallskip
(d) {\it Suppose $U \cong D^{\ga_{n-1}}$.} In this case $\ga_{n-1}$ satisfies
the condition (\ref{sgam}). Hence $\dim U \geq \fs(n)/2$ by 
Proposition \ref{PGamma}, yielding a contradiction again.  
 
\subsubsection{{\bf Case III:} $\pna-\pnb = \pna-\pnc = 1$} 
This case arises precisely when either $p|(n-2)$, or
$p{\not{|}}(n-1)(n-2)(n-3)$ and $2{\not{|}}n$. In particular,
$$\dim V \geq 3f(n-2) \geq 2^{\lfloor \frac{n-1}{2}\rfloor}(3n-15) > 
  2^{\lfloor \frac{n+1}{2}\rfloor}(n-2) \geq f(n).$$ 
Thus we get a contradiction if $a(V) = 0$, or if $\fs(n) = f(n)$.

Hence $a(V) = 1$ and $\fs(n) > f(n)$, i.e. $n$ is even and
$p|(n-2)$; in particular, $\fs(n) = 2^{(n+2)/2}(n-2)$. 
If $n = 16$ then $p = 7$. In this case,
since $d_3(V) \geq d_2(V) \geq 3$, by (\ref{n=13}) we must have
$$\dim V \geq 3d(p,13) \geq 10368 > 7168 = \fs(16),$$
a contradiction. 

So we may assume that $n \geq 20$. We will show that 
each of the large composition factors $U_j$ of $\res_{n-1}V$ has dimension
at least $\fs(n)/2 = 2^{n/2}(n-2)$, leading to the contradiction that 
$\dim V \geq \fs(n)$. Since $n-1 \equiv 1 (\mod p)$, by 
Proposition \ref{JS4} we need to consider the following
three possibilities for $U_j$. 

\smallskip
(a) $d_2(U_j) \geq 3$. Applying the induction hypothesis to 
the large composition factors of $\res_{n-3}U_j$ we get 
$$\dim U_j \geq 3f(n-3)= 2^{(n-2)/2}(3n-15) \geq \fs(n)/2.$$ 

\smallskip
(b) $U_j \cong D^{\ga_{n-1}}$. Recall that $2p|(n-2)$ (in particular 
$n \geq 2p+2$), hence using (\ref{b12}) we have 
$$\dim U_j \geq 8b_{n-2} = 2^{n/2}(2n-10) > \fs(n)/2.$$

\smallskip
(c) $U_j \in \JS(0)$. Applying Proposition \ref{JS0} and the induction 
hypothesis to the large composition factors of $\res_{n-7}U_j$ we get 
$$\dim U_j \geq 24f(n-7) \geq 24 \cdot 2^{(n-8)/2}(n-11) \geq 2^{n/2}(n-2) = \fs(n)/2$$
if $n \geq 29$. Also, if $p \neq 3$, then  
$$\dim U_j \geq 24f(n-7) \geq 24 \cdot 2^{(n-6)/2}(n-10) \geq 2^{n/2}(n-2) = \fs(n)/2.$$
It remains to rule out the cases where $16 \leq n \leq 28$ and
$2p = 6|(n-2)$, i.e. $n = 20$ or $n = 26$. If $n = 20$, then by 
Proposition \ref{JS0} and (\ref{n=13}) we have  
$$\dim U_j \geq 24 \cdot d(p,13) \geq 24 \cdot 3456 > 18432 = \fs(20)/2.$$

Finally, assume $(n,p) = (26,3)$. We claim that any large irreducible
$\T_{19}$-supermodule $X$ has dimension at least $3d(p,13) = 10368$. 
(Indeed, this is certainly true if $d_2(X) \geq 3$ or $d_3(X) \geq 3$. 
If $d_2(X),d_3(X) \leq 2$, then $X \cong D^{\ga_{19}}$ by Proposition \ref{JS4} and
Lemma \ref{JS02}. In this case $\dim X \geq \fs(19) = 15360$ by 
Proposition \ref{PGamma}.) Now applying  
Proposition \ref{JS0} to $U_j$ we get 
$$\dim U_j \geq 24 \cdot 10368 = 248832 > 196608 = \fs(n)/2.$$  
We have completed the proof of the Main Theorem.

\vspace{-1 mm}
\small
\ifx\undefined\bysame
\newcommand{\bysame}{\leavevmode\hbox to3em{\hrulefill}\,}
\fi

\end{document}